\documentclass{amsart}

\usepackage{amssymb}

\newtheorem{theorem}{Theorem}[section]
\newtheorem{lemma}[theorem]{Lemma}
\newtheorem{proposition}[theorem]{Proposition}
\newtheorem{corollary}[theorem]{Corollary}
\theoremstyle{definition}

\theoremstyle{remark}
\newtheorem{remark}[theorem]{Remark}
\numberwithin{equation}{section}

\newcommand{\R}{{\mathbb R}}
\newcommand{\N}{{\mathbb N}}
\newcommand{\Z}{{\mathbb Z}}

\newcommand{\D}{\mathcal{D}}

\newcommand{\e}{\epsilon}
\newcommand{\g}{\gamma}
\newcommand{\B}{\mathcal{B}}

\title[Parabolic Problem]{Porous Medium flow with both a fractional potential pressure and fractional time derivative}
\author{Mark Allen}
\address{Department of Mathematics, The University of Texas at Austin,
Austin, TX 78712, USA}
\email{mallen@math.utexas.edu}

\author{Luis Caffarelli}
\address{Department of Mathematics, The University of Texas at Austin,
Austin, TX 78712, USA}
\email{caffarel@math.utexas.edu}

\author{Alexis Vasseur}
\address{Department of Mathematics, The University of Texas at Austin,
Austin, TX 78712, USA}
\email{vasseur@math.utexas.edu}

\begin{document}

\begin{abstract}
 We study a porous medium equation with right hand side. The operator has nonlocal diffusion effects given by an inverse fractional Laplacian operator. 
 The derivative in 
 time is also fractional of Caputo-type and which takes into account ``memory''. The precise model is
  \[
   \D_t^{\alpha} u - \text{div}(u(-\Delta)^{-\sigma} u) = f, \quad 0<\sigma <1/2.
  \]
 We pose the problem over $\{t\in \R^+, x\in \R^n\}$ with nonnegative initial data $u(0,x)\geq 0 $ as well as right hand side $f\geq 0$. 
 We first prove  
 existence for weak solutions when $f,u(0,x)$ have exponential decay at infinity.  Our main result is H\"older continuity for such weak solutions. 
\end{abstract}

\maketitle

\section{Introduction}
 In this paper we study both existence and regularity for solutions to a porous medium equation. The pressure is related to the density via a nonlocal
 operator. This diffusion takes into account long-range effects. The time derivative is nonlocal and fractional and therefore takes into account the past. In 
 the typical derivation of the porous medium equation (see \cite{v07}) the equation one considers is 
  \[
   \partial_t u + \text{div} (\textbf{v}u)=0,
  \] 
 with $u(t,x)\geq 0$. By Darcy's law in a porous medium  $\textbf{v}= -\nabla p$ arises as a potential where $p$ is the pressure. According to a state law
 $p=f(u)$. In our case we consider a potential which takes into account long range interactions, namely $p=(-\Delta)^{-\sigma}u$. A porous medium equation with
 a pressure of this type 
  \begin{equation}   \label{e:tloc}
   \partial_t u = \text{div}(u(-\Delta)^{-\sigma}u)
  \end{equation}
 has been recently studied. For $0<\sigma<1$ with $\sigma \neq 1/2$, existence of solutions was shown in \cite{cv11} 
 while regularity and further existence properites were studied in \cite{cfv13}. 
 Uniqueness for the range $1/2\leq \sigma <1$ was shown in \cite{zxc14}. Another model of the porous medium equation
  \[
   D_t^{\alpha} u - \text{div}(\kappa (u)Du)=f
  \]
 was introduced by Caputo in \cite{c99}. In the above equation $D_t^{\alpha}$ is the Caputo derivative and the diffusion is local. 
 Solvability for a more general equation was recently studied in 
 \cite{z12}. The fractional derivative takes into account models in which there is ``memory''. The Caputo derivative has also been recently shown
 (see \cite{d04}, \cite{d05} ) to be 
 effective in modeling problems in plasma transport. See also \cite{mk00} and \cite{z02} for further models that utilize fractional equations
 in both space and time to account for long-range interactions as well as the past.

 The specific equation we study is 
  \begin{equation}  \label{e:e}
    \D_t^{\alpha} u(t,x) - \text{div} \left(u \nabla (-\Delta)^{-\sigma} u \right)=f(t,x).
  \end{equation}
 The operator $\D_t^{\alpha}$ is of Caputo-type and is defined by 
  \[
   \D_t^{\alpha} u := \frac{\alpha}{\Gamma(1-\alpha)}\int_{-\infty}^t [u(t,x)-u(s,x)]K(t,s,x) \ ds.  
  \]
 When $K(t,s,x)=(t-s)^{-1-\alpha}$ this is exactly the Caputo derivative - see Section \ref{s:caputo} - which we denote by $D_t^{\alpha}$. 
 We assume the following bounds on the kernel $K$ 
  \begin{equation}  \label{e:kernelb}
   \frac{1}{\Lambda(t-s)^{1+\alpha}} \leq \frac{\alpha}{\Gamma(1-\alpha)} K(t,s,x) \leq \frac{\Lambda}{(t-s)^{1+\alpha}}
  \end{equation}
 Our kernel in time then can be thought of as having ``bounded, measureable coefficients''. We also require the following relation on the kernel
  \begin{equation}  \label{e:kernel}
   K(t,t-s)=K(t+s,t).
  \end{equation}
 The relation \eqref{e:kernel} allows us to give a weak - in space and in time - formulation of \eqref{e:e}. This weak formulation is given 
 in Section \ref{s:caputo}. 
 
 In this paper we also restrict ourselves to the range $0<\sigma<1/2$. In \cite{cfv13} use of a transport term was made to work in the range $1/2<\sigma<1$. 
 We have not yet found the correct manner in which to prove our results for $1/2<\sigma<1$ when dealing with the nonlocal fractional time derivative $\D_t^{\alpha}$.

 \subsection{Accounting for the Past.}
  Nonlocal equations are effective in taking into account long-range interactions and taking into account the past. However, 
  the nonlocal aspect of the equation provides both advantages and disadvantages in studying local aspects of the equation. One advantage is that there is
  a relation between two points built into the equation. Indeed, we utilize two nonlocal terms that are not present in the classical porous medium equation
  to prove Lemma \ref{l:2down}. One disadvantage of nonlocal equations is that when rescaling of the form $v(t,x)=Au(Bt,Cx)$, the far away portions of $u$ cannot
  be discarded, and hence $v$ begins to build up a ``tail''. Consequently, the usual test function $(u-k)_+$ or $F((u-k)_+)$ for some function 
  $F$ and a constant $k$ is often insufficient. One must instead consider $F((u-\phi)_+)$ where $\phi$ is constant close by but has some ``tail'' growth at infinity. 
  This difficulty of course presents itself with the Caputo derivative. One issue becomes immediately apparent. If we choose $F((u-\phi)_+)$ as  a test function, then 
   \[
    \begin{aligned}
    \int_a^T F((u-\phi)_+) D_t^{\alpha} u \ dt &= \int_a^T F((u-\phi)_+) D_t^{\alpha} ((u-\phi)_+ - (u-\phi)_- ) \ dt \\
                                               &\quad + \int_a^T F((u-\phi)_+) D_t^{\alpha} \phi \ dt. 
    \end{aligned}
   \]
  The second term will no longer be identically zero if $\phi$ is not constant. When using energy methods, this second term can be treated as part of the right 
  hand side, and hence it becomes natural to consider an equation of the form \eqref{e:e} with a right hand side. The main challenge with accomodating a nonzero
  right hand side is that the natural test function $\ln u$ used in \cite{cv11} and \cite{cfv13} is no longer available since the function $u$ can evaluate zero. 
  Indeed if the initial data for a solution is compactly supported, then the solution is compactly supported for every time $t>0$ (see Remark \ref{r:finite}). 
  We choose as our basic
  test function $u^{\gamma}$ for $\gamma>0$. For $\sigma$ small we will have to choose $\gamma$ small. 
  We then can accomodate a right hand side as well as avoiding delicate integrability issues involved when using $\ln u$
  as a test function. Using careful analysis, it is still most likely possible to utilize $\ln u$ as a test function for our equation \eqref{e:e} with
  zero right hand side, but we find it more convenient to use $u^{\gamma}$ and prove the stronger result that includes a right hand side. Our method using $u^{\g}$
  should also work for the equation \eqref{e:tloc} to be able to prove existence and regularity with a right hand side. 
  One benefit of accomodating a right hand side in $L^{\infty}$
  is that we obtain immediately regularity up to the initial time for smooth initial data, see Theorem \ref{t:continuity}.

 \subsection{Overview of the Main results}
  We will prove our results for a class of weak solutions \eqref{e:main} later formulated in Section \ref{s:caputo}.  
  Our first main result is existence. We use an approximating scheme as in \cite{cv11} as well as discretizing in time as in \cite{acv15}. We prove
   \begin{theorem}   \label{t:existence}
    Let $0\leq u_0(x), f(t,x) \leq Ae^{-|x|}$ for some $A \geq 0$. Then there exists a solution $u$ to \eqref{e:main} in $(0,\infty)$
    that has initial data $u(0,x)=u_0(x)$. 
   \end{theorem}
  
   \begin{remark}  \label{r:unique}
    Our constructions are made via recursion over a finite time interval $(0,T_1)$. Since our constructions are made via recursion, 
    if $T_2 =mT_1$ for $m \in \N$, it is immediate that if $u_i$ is the solution constructed on $(0,T_i)$, then $u_2 = u_1$ on $(0,T_1)$. 
   \end{remark}
   
   \begin{remark}  \label{r:n1}
    For technical reasons seen in the proof of Lemma \ref{l:triebel}, when $n=1$ we make the further restriction $0<\sigma<1/4$. 
   \end{remark}
  
  The main result of the paper is an interior H\"older regularity result. As expected the H\"older norm will depend on the distance from the interior domain
  to the initial time $t_0$. However, if we assume the intial data $u_0$ is regular enough - say for intance $C^2$ - then we obtain regularity up to the 
  initial time. This is a benefit of allowing a right hand side. By extending the values of our solution $u(t,x)=u(0,x)$ for $t<0$, we satisfy \eqref{e:main}
  on $(-\infty, \infty) \times \R^n$ with a right hand side in $L^{\infty}$. The right hand side $f$ for $t \leq 0$ will not necessarily satisfy $f \geq 0$; however,
  this nonnegativity assumption on $f$ was only necessary to guarantee the existence of a solution $u \geq 0$. It is not a necessary assumption to prove regularity. 
  From Remark \ref{r:unique} the solution constructed on $(-\infty, T)$ will agree with the original solution over the interval $(0,T)$. 
  
   \begin{theorem}  \label{t:continuity}
    Let $u$ be a solution to \eqref{e:main} obtained via approximation from Theorem \ref{t:existence} on $[0,T] \times \R^n$ with
    $0\leq u_0(x), f(t,x) \leq Ae^{-|x|}$. Assume also $u_0 \in C^2$. 
    Then $u$ is $C^{\beta}$ continuous on $[0,T]\times \R^n$- for some exponent $\beta$ depending on $\alpha, \Lambda, n, \sigma$ - 
     with a constant that depends on the $L^{\infty}$ 
    norm of $u$ and $f$ and  $C^2$ norm of $u_0$.  
   \end{theorem}

 \subsection{Future Directions}
  We prove existence and regularity for solutions obtained via limiting approximations. In this paper we do not address the issue of uniqueness. As mentioned earlier, 
  uniqueness for \eqref{e:tloc} for the range $1/2\leq \sigma <1$ was shown in \cite{zxc14}. The issue of uniqueness for \eqref{e:e} is not trivial
  because of the nonlinear aspect of the equation as well as the lack of a comparison principle. The equation \eqref{e:main} which we consider also should present new 
  difficulties because of the weak/very-weak formulation in time as well as the minimal ``bounded, measurable'' assumption \eqref{e:kernelb} on the kernel $K(t,s,x)$. 
  An interesting problem would be to then address the issue of uniqueness for solutions of \eqref{e:main}. 
  
  Theorems \ref{t:existence} and \ref{t:continuity} can most likely be further refined by making less assumptions on $u(0,x)$, 
  assuming a right hand side $f \in L^p$ as was done for a similar problem in \cite{z12}, and proving the estimates uniform as $\sigma \to 0$ and 
  recoving H\"older continuity for the local diffusion problem. Also, as mentioned earlier the Theorems can be improved to include the range
  $1/2\leq \sigma <1$. 
  
  Finally, just like in the local porous medium equation \cite{v07} as well as for \eqref{e:tloc}, the equation \eqref{e:main} has the property of finit propagation,
  see Remark \ref{r:finite}. 
  Therefore, it is of interest to study the free boundary $\partial \{u(t,x)>0\}$. 
      
 \subsection{Outline} 
  The outline of this paper will be as follows. In Section \ref{s:caputo} we state basic results for the Caputo derivative. We also give the weak formulation of
  the equation we study. In Section \ref{s:discretize} we state some results for the discretized version of $\D_t^{\alpha}$ that we will use to prove the existence
  of solutions. In Section \ref{s:existence} we follow the method approximation method and use the estimates from \cite{cv11} combined with the method of discretization
  and the estimates presented in \cite{acv15} to prove existence. In Section \ref{s:lemmas} we state 
  the main Lemmas that we will need to be able to prove H\"older regularity. In Section \ref{s:pullup} we prove the most technically difficult Lemma \ref{l:pullup}
  of 
  the paper. This Lemma \ref{l:pullup} most directly handles the degenerate nature of the problem. In Section \ref{s:pulldown} we prove an analogue of \ref{l:pullup}.
  In Section \ref{s:oscillation} we prove the final Lemmas we need which give a one-sided decrease in oscillation from above. The one-sided decrease in oscillation
  combined with Lemma \ref{l:pullup} is enough to prove the H\"older regularity and this is explained in Section \ref{s:regularity}.

 \subsection{Notation}
  We list here the notation that will be used consistently throughout the paper. The following letters are 
  fixed throughout the paper and always refer to:
   \begin{itemize}

   \item  $\alpha$ - the order of the Caputo derivative.
   
   \item  $\sigma$ - the order of inverse fractional Laplacian $(-\Delta)^{-\sigma}$.  We use $\sigma$
    for the order because $s$ will always be a variable for time. 
   
   \item $a$ -  the initial time for which our equation is defined.
   
   \item $D_t^{\alpha}$ - the Caputo derivative as defined in Section \ref{s:caputo}.
   
   \item $\D_t^{\alpha}$ - the Caputo-type fractional derivative with ``bounded, measurable'' coefficients with bounds \eqref{e:kernelb} and
   relation \eqref{e:kernel}. 
   
   \item $\Lambda$ the constant appearing in \eqref{e:kernelb}.
      
   \item $\D_{\e}^{\alpha}$ - the discretized version of $\D_t^{\alpha}$ as defined in \eqref{e:dc} 
   
   \item $\e$ - will always refer to the time length of the discrete approximations 
   as defined in Section \ref{s:discretize}
   
   \item $n$ - will always refer to the space dimension. 
  
   \item $\Gamma_m$ - the parabolic cylinder $(-m,0)\times B_m$. 
   
   \item $W^{\beta,p}$ - the fractional Sobolev space as defined in \cite{dpv12}. 
   \item $u_{\pm}$ - the positive and negative parts respectively so that $u = u_+ - u_-$. 
   \item $\tilde{u}$ -  the extension $\tilde{u}(t)=u(\e j)$ for $\e j-1 < t, \e j$.        
 \end{itemize}

\section{Caputo Derivative} \label{s:caputo}
 In this section we state various properties of the Caputo derivative that will be useful.
 The Caputo derivative for $0<\alpha<1$ is defined by 
  \[
   _{a}D_{t}^{\alpha} u(t) := \frac{1}{\Gamma(1-\alpha)} \int_{a}^{t} \frac{u'(s)}{(t-s)^{\alpha}} \ ds
  \]
 By using integration by parts we have 
  \begin{equation} \label{e:fracalt}
    \Gamma(1-\alpha) \ _{a}D_{t}^{\alpha} u(t) = 
   \frac{u(t)-u(a)}{(t-a)^{\alpha}} + \alpha \int_{a}^t \frac{u(t)-u(s)}{(t-s)^{\alpha + 1}} \ ds.
  \end{equation}
 For the remainder of the paper we will drop the subscript $a$ when the initial point is understood. 
 We now recall some properties of the Caputo derivative that were proven in \cite{acv15}.
  

  
 
 For a function $g(t)$ defined on $[a,t]$, it is advantageous to define  $g(t)$ for $t<a$. Then we have the formulation 
  \[
   _a D_t^{\alpha} g(t) = \  _{-\infty}D_t^{\alpha} g(t)
     = \frac{\alpha}{\Gamma(1-\alpha)} \int_{-\infty}^t \frac{g(t)-g(s)}{(t-s)^{1+\alpha}},
  \]
 utilized in \cite{acv15}. (See also \cite{bmst15} for properties of this one-sided nonlocal derivative.)
 This looks very similar to $(-\Delta)^{\alpha}$ except the integration only occurs for $s<t$. In 
 this manner the Caputo derivative retains directional derivative behavior while at the same time sharing
 certain properties with $(-\Delta)^{\alpha}$. This is perhaps best illustrated by the following
 integration by parts formula for the Caputo derivative   
 
 \begin{proposition}  \label{p:changev}
  Let $g,h\in C^1(a,T)$. Then 
   \begin{equation}  \label{e:changev}
    \begin{aligned}
    \int_{a}^T g D_t^{\alpha}h + h D_t^{\alpha} g &= 
     \int_{a}^T {g(t)h(t) \left[\frac{1}{(T-t)^{\alpha}} + \frac{1}{(t-a)^{\alpha}} \right]} \ dt \\
     & \quad + \alpha \int_{a}^T \int_{a}^t \frac{[g(t)-g(s)][h(t)-h(s)]}{(t-s)^{1+\alpha}} \ ds \ dt \\
     & \quad -   \int_{a}^T \frac{g(t)h(a)+h(t)g(a)}{(t-a)^{\alpha}} \ dt.
    \end{aligned}
   \end{equation}
 \end{proposition}
 
 Formula \eqref{e:fracalt} is based on the following formal computation
  \[
   \begin{aligned}
     \int_a^T \int_a^t \frac{g(t)-g(s)}{(t-s)^{1+\alpha}} \ ds \ dt
     &=  \int_a^T \int_a^t \frac{g(t)}{(t-s)^{1+\alpha}} \ ds \ dt
       - \int_a^T \int_a^t \frac{g(s)}{(t-s)^{1+\alpha}} \ ds \ dt \\
     &=  \int_a^T \int_a^t \frac{g(t)}{(t-s)^{1+\alpha}} \ ds \ dt 
       - \int_a^T \int_s^T \frac{g(s)}{(t-s)^{1+\alpha}} \ dt \ ds \\
     &=  \int_a^T \int_a^t \frac{g(t)}{(t-s)^{1+\alpha}} \ ds \ dt 
       - \int_a^T \int_t^T \frac{g(t)}{(s-t)^{1+\alpha}} \ ds \ dt \\
     &= \int_a^T g(t)\left( \int_0^{t-a} \frac{ds}{s^{1+\alpha}} 
        - \int_{0}^{T-t} \frac{ds}{s^{1+\alpha}} \right) \ dt \\
     &= \alpha^{-1}\int_a^T g(t) \left(\frac{1}{(T-t)^{\alpha}} - \frac{1}{(t-a)^{\alpha}} \right) \ dt.
   \end{aligned}
  \]
 In the above computation to utilize the cancellation we only need a kernel $K(t,s)$ satsifying
  \[
   K(t,t-s)=K(t+s,t). 
  \] 
 To make the above computation rigorous we will use the discretization in Section \ref{s:discretize}. 
 An alternative, equivalent integration by parts formula is to extend $g(t)=g(a)$ for $t<a$. Then for any
  $h\in C^1$ with $h(t)=0$ for any $t<b$ for some $b$ we have 
  \begin{equation}   \label{e:crhs}
   \begin{aligned}
    \int_{-\infty}^T h(t) D_t^{\alpha} g(t) \ dt
     &= c_{\alpha} \int_{-\infty}^T \int_{-\infty}^{2t-T} \frac{[g(t)-g(s)][h(t)-h(s)]}{(t-s)^{1+\alpha}} \ ds \ dt \\
     &\quad + c_{\alpha} \int_{-\infty}^T \int_{-\infty}^t \frac{g(t)h(t)}{(t-s)^{1+\alpha}} \ ds \ dt \\
     &\quad - \int_{-\infty}^T g(t) D_t^{\alpha} h(t) \ dt,
   \end{aligned}
  \end{equation} 
 with $c_{\alpha}=\alpha \Gamma(1-\alpha)^{-1}$. 
 Now \eqref{e:changev} and \eqref{e:crhs} will both imply each other, so \eqref{e:crhs} is an alternative way
 of handling the initial condition $g(a)$. Furthermore, both \eqref{e:changev} and \eqref{e:crhs} are weak formulations
 for the Caputo derivative that only require that $g \in H^{\alpha/2}((a,T))$. 
 \eqref{e:crhs} will work for any kernel $K(t,s)$ satisfying the relation \eqref{e:kernel}. 
 In view of \eqref{e:crhs} 
 we now give the exact formulation of our weak solutions. We assume the bounds \eqref{e:kernelb} and the
 relation \eqref{e:kernel} on the kernel $K(t,s,x)$. 
 For smooth initial data $u_0 \in C^2$, we assign $u(t,x)=u(a,x)$ for $t<a$. Then as stated earlier
 in the Introduction, for $t\leq a$,  a solution $u$ will have right hand side 
  \[
   \text{div}(u_0 (-\Delta)^{-\sigma} u_0) \in L^{\infty}. 
  \] 
 We say that $u$ is a weak solution if for any $\phi \in C_0^{\infty} (-\infty,T)\times \R^n$ we  have
  \begin{equation}  \label{e:main}
    \begin{aligned}
     &\int_{\R^n} \int_{-\infty}^T \int_{-\infty}^t [u(t,x)-u(s,x)][\phi(t,x)-\phi(s,x)]K(t,s,x) \ ds \ dt \ dx \\
     &\quad + \int_{\R^n} \int_{-\infty}^T \int_{-\infty}^{2t-T} u(t,x)\phi(t,x) K(t,s,x) \ ds \ dt \ dx \\
     &\quad -\int_{\R^n} \int_{-\infty}^T u(t,x) \D_t^{\alpha} \phi(t,x) \ dt \ dx \\
     &\quad +\int_{-\infty}^T \int_{\R^n} \nabla \phi(t,x) u(t,x) \nabla (-\Delta)^{-\sigma} u \ dx \ dt\\
     &= \int_{-\infty}^T \int_{B_R} f(t,x) \phi(t,x).
    \end{aligned} 
  \end{equation}

 We will also utilize a fractional Sobolev norm that arises from the fractional derivative. 
 \begin{lemma}  \label{l:ext}
  Let $u$ be defined on $[a,T]$. We have for two constants $c_1,c_2$ depending on $\alpha,|T-a|$
   \[
    \begin{aligned}
    \| u \|_{L^{\frac{2}{1-\alpha}}}(a,T)   & \leq c_1 \|u \|_{H^{\alpha/2}}^2(a,T) \\
         &\leq 
    c_2 \left( \alpha \int_{a}^T \int_{a}^t \frac{|u(t)-u(s)|^2}{|t-s|^{1+\alpha}} \ ds \ dt +
    \int_a^T \frac{u^2(t)}{(T-a)^{\alpha}} \right)
    \end{aligned}
   \]
 \end{lemma}


 The following estimate will be needed for our choice of cut-off functions. 
 \begin{lemma}  \label{l:fracbound}
  Let 
   \[
     h(t):= \max \{ |t|^{\nu} -1 , 0 \} \\
   \] 
    with $\nu < \alpha$. Then
   \[
    \D_t^\alpha h \geq  -c_{\nu,\alpha, \Lambda}
   \]
  for $t \in \R$. Here, $c_{\nu,\alpha,\Lambda}$ is a constant depending only on $\alpha , \nu, \Lambda$. 
 \end{lemma}
 
 Finally, we point out that if $g=g_+ - g_-$ the positive and negative parts respectively, then 
  \begin{equation}  \label{e:posneg}
   \int_a^T g_{\pm}(t)\D_t^{\alpha} g_{\mp}(t) \geq 0. 
  \end{equation} 
  

\section{Discretization in time}  \label{s:discretize}
 To prove existence of solutions to \eqref{e:main} we will discretize in time. 
 The discretization also allows us to make the computations
 involving the fractional derivative rigorous.  This section contains properties of a discrete fractional
 derivative which we will utilize.
 
 
 For future reference we denote the discrete Fractional derivative with kernel $K$ as
 \begin{equation}  \label{e:dc}
  \D_{\e}^{\alpha} u(a+\e j):=   \epsilon 
    \sum_{-\infty<i <j} [u(a+\e j) - u(a+\e i)]K(\e j, \e i)
 \end{equation}
 
 The following is the discretized argument for the cancellation that 
 appears in the formal computation of the version of \eqref{e:fracalt} that
 we will need.  
  \begin{lemma}  \label{l:cancel}
   Assume $g(a)=0$ and define $g(t):=0$ for $t<a$. 
   Assume relation \eqref{e:kernel} on $K$. Then 
    \[
      \e^2 \sum_{j\leq k} \sum_{i<j} [g(a+\e j)-g(a+\e i)]K(a+\e j, a+ \e i) 
      = \e^2 \sum_{j\leq k} \ \sum_{i<2j-k-1} g(\e j)K(\e j,  \e i). \\
    \]
   Let $\tilde{g}(t):=g(\e j)$ for $\e j-1< t \leq \e j$. If $g\geq 0$, then there exists $c$
   depending only on $\alpha, \Lambda$ such that if $\e < 1$, then
    \[
      \e^2 \sum_{j\leq k} \sum_{i<j} [g(a+\e j)-g(a+\e i)]K(a+\e j, a+ \e i) 
      \geq c \int_{-\infty}^T \frac{\tilde{g}(t)}{(T-t)^{\alpha}} \ dt.
    \]
  \end{lemma}
 
  \begin{proof}
   For notational simplicity we assume $a=0$. 
   \[
     \begin{aligned}
      &\e^2 \sum_{j\leq k} \sum_{i<j} [g(\e j)-g(\e i)]K(\e j, \e i)  \\
       &=   \e^2 \sum_{j\leq k} \sum_{i<j} g(\e j)K(\e j,  \e i) 
           - \e^2 \sum_{j\leq k} \sum_{i<j} g(\e i)K(\e j,  \e i) \\
       &=  \e^2 \sum_{j\leq k} \sum_{i<j} g(\e j)K(\e j,  \e i) 
           - \e^2 \sum_{i< k} \sum_{i<j\leq k} g(\e i)K(\e j,  \e i) \\
       &=  \e^2 \sum_{j\leq k} \sum_{i<j} g(\e j)K(\e j,  \e i) 
           - \e^2 \sum_{j< k} \sum_{j<i\leq k} g(\e j)K(\e i,  \e j) \\
       &= \e^2 \sum_{j\leq k} \sum_{i<2j-k-1} g(\e j)K(\e j,  \e i) 
     \end{aligned}
    \]
   The second inequality follows from the estimates in \cite{acv15}. 
  \end{proof}
 
 Lemma \ref{l:cancel} combined with the estimates in \cite{acv15} can be used to show.
  \begin{lemma}  \label{l:1byparts}
   Let $u(0)=0$ and assume $u\geq 0$. For fixed $0<\e<1$, let $\tilde{u}$ be the extension defined as in Lemma \ref{l:cancel}. 
   Let $K$ satisfy conditions \eqref{e:kernelb} and \eqref{e:kernel}. Then there exists two constants $c_1,c_2$ depending only on $\alpha$
   such that
    \[
     \e^2 \sum_{j\leq k} \sum_{i<j} u(\e j)[u(\e j)-u(\e i)]K(\e j, \e i)
      \geq c \| u \|_{H^{\alpha/2}}^2 \geq c \| u \|_{H^{\alpha/2}}^2c \left(\int_{-\infty}^T u^{\frac{2}{(1-\alpha)}} \right)^{1-\alpha} = 
    \]
   where $c$ depends on $\alpha$ and $\Lambda$. 
  \end{lemma}

 This next lemma is analogous to Lemma \ref{l:fracbound} and was shown in \cite{acv15}. 
  \begin{lemma}  \label{l:dfracbound}
   Let $h$ be as in Lemma \ref{l:fracbound}. Then for $0<\epsilon < 1$ there exists $c_{\nu,\alpha}$
   depending on $\alpha$ and $\nu$ but 
   independent of $a$ such that 
    \[
     \D_{\epsilon}^{\alpha} h(t) \geq -c_{\nu,\alpha}
    \] 
   for $t\in \epsilon \Z$ and $a<t<0$.
  \end{lemma}

 This last estimate we will use often 
  \begin{lemma}  \label{l:discconvex}
   Let $F$ be a convex function with $F'' \geq \gamma, F' \geq 0, F(0)=0$. Assume $g\geq 0, g(a)=0$. Then there exists $c$ depending on $\alpha, \Lambda$
    such that 
    \[
     \e \sum_{j\leq k} F(g(\e j)) \D_{\e}^{\alpha}g(\e j) \geq c\e \sum_{j\leq k} \frac{F(g(\e j))}{(\e(j-i))^{1+\alpha}}
      + c \frac{\g}{2}\e^2  \sum_{j\leq k}\sum_{e<j} \frac{[g(\e j)-g(\e i)]^2}{(\e(j-i))^{1+\alpha}}
    \]
  \end{lemma}

  \begin{proof}
   Since $F$ is convex, 
    \[
     F(g(\e j))[g(\e j)-g(e i)] \geq F(g(\e j))- F(g(\e i)) + \frac{\g}{2} [g(\e j)- g(\e i)]^2. 
    \]
   The result then follows from applying Lemma \ref{l:cancel}. 
  \end{proof}

 Finally we point out that if 
 $g$ is a limit of $\tilde{g}_{\e}$ which are discretized problems with the assumptions in Lemma \ref{l:discconvex} it follows that 
  \begin{equation} \label{e:convex}
   \int_a^T F(g(t)) \D_t^{\alpha} g(t) \ dt \geq c \int_a^T \frac{F(g(t))}{(T-a)^{\alpha}} 
    + c \frac{\g}{2}\int_a^T \int_a^t \frac{[g(t)-g(s)]^2}{(t-s)^{1+\alpha}} \ ds \ dt. 
  \end{equation}

\section{Existence}   \label{s:existence}
 In this section we prove the existence of weak solutions following the construction given in \cite{cv11}. We will also  
 discretize in time. We first consider 
 a smooth approximation of the kernel $(-\Delta)^{-\sigma}$ as $K_{\zeta}$. We start with the smooth classical solution to the elliptic problem
  \begin{equation}  \label{e:aprox}
   g u - \delta \text{div}((u+d) \nabla u) - \text{div}((u + d) \nabla K_{\zeta} u ) = f \text{  on  }  B_R
  \end{equation}
 with $u \equiv 0$ on $\partial B_R$. $g,f\geq 0$ and smooth. The sign of $f,g$ guarantees the solution is nonnegative. $\delta, d>0$ are constants. 
 The nonlocal part is computed in the expected way by extending $u=0$ on $B_R^c$.  To find such a solution
 we first consider the linear problem 
   \[
   g u - \delta\text{div}((v+d)\nabla u)  - \text{div}((v + d) \nabla K_{\zeta} u ) = f \text{  on  }  B_R
   \]
  for $v \in C_0^{0,\beta}$ with $v\geq 0$.
  With fixed  $d,\delta, R, \zeta >0$ one can apply Schauder estimate theory to conclude 
   \[
    \| u \|_{C_0^{1,\beta}} \leq C \| v \|_{C_0^{0,\beta}} 
   \]
  with $u \geq 0$. The map $T: v \to u$ is then a compact map. The set $\{v\}$ with $v \geq 0$ and $v \in C_0^{0,\beta}$ is a closed convex set, 
  and hence
  we can apply the fixed point theorem (Corollary 11.2 in \cite{gt01}), to conclude there is a solution to \eqref{e:aprox}. 
  By bootstrapping we conclude $u$ is smooth.  
  
  
 
 Now we use the existence of solutions to \eqref{e:aprox} to obtain - via recursion - solutions to the discretized problem
  \begin{equation}  \label{e:daprox}
   \D_{\e}^{\alpha} u - \delta \text{div}((u+d)\nabla u) - \text{div}((u + d) \nabla K_{\zeta} u ) = f \text{  on  }  [0,T]\times B_R.
  \end{equation}
   with $u(0,x)=u_0(x)$ an initially defined smooth function with compact support. $\e=T/k$ for some $k\in \N$. 
   We will eventually let $k \to \infty$, so that $\e \to 0$. 
   
 For the next two Lemmas we will utilize the solution to  
  \begin{equation}  \label{e:ode}
   D_t^{\alpha} Y(t) = cY(t)+h(t)
  \end{equation}
 which as in \cite{d10} is given by  
  \[
   Y(t)= Y(0)E_{\alpha}(ct^{\alpha})+\alpha\int_{0}^t(t-s)^{\alpha-1}E_{\alpha}'(c(t-s)^{\alpha})h(t) \ dt, 
  \]
 where $E_{\alpha}$ is the Mittag-Leffler function of order $\alpha$. 
 We will utilize in the next Lemmas two specific instances of \eqref{e:ode}. We define $Y_1(t)$ to be the solution to \eqref{e:ode} with 
 $Y(0)=\sup u(0,x), c=0, h=2\Lambda f$. We define $Y_2(t)$ to be the solution to \eqref{e:ode} with  
 $c=C\Lambda^{-1},Y_{2}(0)=2$.
   and $h=0$. The constant $C$ will be chosen later. 
  \begin{lemma}  \label{l:boundab}
   Let $u$ be a solution to \eqref{e:daprox}. Let $Y_1(t)$ be defined as above.
   Then there exists $\e_0$ depending only on $T,\alpha,\|f\|_{L^{\infty}}$ such that if $\e \leq \e_0$, then 
    \[
     u(\e j) \leq Y_1(\e j)
    \]
  \end{lemma}  
  
  \begin{proof}
   Since $Y_1(t)$ is an increasing function 
    \[
     \D_{\e}^{\alpha} Y(\e j) \geq \Lambda^{-1} D_{\e}^{\alpha} Y(\e j).
    \]
   Depending on $T$ and $\|f\|_{L^{\infty}}$, there exists $\e_0$ such that if $\e\leq \e_0$, then 
    \[
      \D_{\e}^{\alpha} Y(\e j) \geq \Lambda^{-1} D_{\e}^{\alpha} Y(\e j)  
      \geq \frac{2}{3\Lambda} D_{t}^{\alpha} Y(\e j)= \frac{4}{3}f  
    \]
   We use $(u(t,x)-Y(t))_+$ as a test function. Since $(u-Y)_+(0)=0$ it follows from \eqref{e:posneg} and Lemma \ref{l:discconvex} that 
    \[
     \e \sum_{j<k} (u-Y)_+ \D_{\e}^{\alpha} [(u-Y)_+ -(u-Y)_-] \geq 0. 
    \] 
    We define
     \[
      \B_{\zeta}(u,v) = \int_{\R^n}\int_{\R^n} H_{\zeta}(x,y)[u(x)-u(y)][v(x)-v(y)] \ dx \ dy, 
     \]
    where $H_{\zeta}(x,y)=\Delta K_{\zeta}$. 
   We have the identity \cite{cfv13}
    \[
     \B_{\zeta}(u,v) = \int_{\R^n}\int_{\R^n} \nabla u(x) \nabla K_{\zeta} v(y) \ dx \ dy.
    \]
   Then for $\e$ small enough and $j>0$
    \[
     \begin{aligned}
       &\e\sum_{j\leq k}\int_{B_R} f(\e j,x)(u(\e j, x)-Y(\e j))_+ \\
       &=\int_{B_R} \e \sum_{j\leq k}(u-Y)_+\D_{\e}^{\alpha}[(u-Y)_+ - (u-Y)_-+Y]  \\
       &\quad + \e\sum_{j\leq k} \int_{B_R} (u+d) \nabla (u-Y)_+ \nabla u \\
       &\quad +\e\sum_{j\leq k} \int_{B_R}(u+d) \nabla (u-Y)_+ \nabla K_{\zeta} u \\
       &\geq \int_{B_r}\e \sum_{j\leq k}\frac{4}{3}f(u-Y)_+ + \int_{\R^n}(u+d) \chi_{\{u>Y\}}\nabla u \nabla K_{\zeta} u \\
       &=\int_{B_R}\e \sum_{j\leq k}\frac{4}{3}f(u-Y)_+  \\
       &\quad + \frac{\e}{2}\sum_{j\leq k} \B_{\zeta}(\chi_{\{u>Y\}} (u+d)^2, u) \\
       &\geq \frac{4}{3} \e \sum_{j\leq k}\int_{B_R}f(\e j,x)(u(\e j,x)-Y(\e j))_+.  \\
     \end{aligned}
    \] 
    Thus $(u-Y)_+ \equiv 0$. 
  \end{proof}
  
  \begin{lemma}  \label{l:tail}
   Let $u$ be a solution to \eqref{e:daprox} in $[0,T]\times B_R$. 
   Assume that 
    \begin{equation}  \label{e:exponent}
       0\leq u(0,x)\leq Ae^{-|x|}, f \leq Ae^{-|x|}. 
    \end{equation}
   If $A$ is large there exists constants 
     \begin{gather*}
      \mu_0 , \delta_0, \zeta_0    \text{ depending only on } R,n,\sigma \\
      \e_0      \text{  depending only on  } T,\alpha, \Lambda \\
      C         \text{  depending only on  } n,\sigma, \| u(0,x)\|_{L^{\infty}}, \| f\|_{L^{\infty}} 
     \end{gather*}
   such that if $\mu < \mu_0, \delta < \delta_0, \zeta < \zeta_0, \e<\e_0$, and $Y_2(t)$ is the solution defined earlier with constant $C$ given above, 
   then
    \[
     0 \leq u(\e j,x) \leq AY_2(\e j)e^{-|x|}, 
    \]
   for any $t=\e j$. 
  \end{lemma}
 
  \begin{proof} 
   As before there exists $\e_0$ depending only on $T, \alpha$ such that for $\e \leq \e_0$ we have
    \[
     \D_{\e}^{\alpha} Y_2(\e j) \geq \frac{2}{3\Lambda} D_{t}^{\alpha} Y_2(\e j).  
    \]
   Since $u$ is smooth and hence continuous, $u \leq LY_2(\e j)e^{-|x|}$ for some $L>A$. 
   We lower $L\geq A$ until it touches $u$ for the first time. 
   Since $u=0$ on $\partial B_R$
   this cannot happen on the boundary. Since $u$  is smooth this 
   cannot happen at a point $(\e j, 0)$. Also, $LY_2 \geq 2A\geq 2u(0,x)$, so this cannot occur at the initial time.  
   We label a point of touching as $(t_c,r_c)$. We compute the operator in nondivergence form 
   and write $K_{\zeta}(u)=p$ and use 
   the estimates in \cite{cv11} to conclude for $\e$ small enough that 
    \[
     \begin{aligned}
      \frac{2}{3}CLY_2(t_c)e^{-r_c} 
        &= \frac{\Lambda}{3}L D_t^{\alpha} Y_2(t_c )e^{-r_c} \\
        &\leq \D_{\e}^{\alpha} LY_2(t_c)e^{-r_c} \\
        &\leq  \D_{\e}^{\alpha} u(t_c) \\
        &= \delta \text{div}((u+d)\nabla u) + \text{div}((u+d) \nabla K_{\zeta} u) +f(t_c,r_c)\\
        &= \delta 2[LY_2(t_c)e^{-r_c}]^2 + \delta d LY_2(t_c)e^{-r_c} \\
        &\quad  -LY_2(t_c)e^{-r_c} \overline{\partial_r p}  +(LY_2(t_c)e^{-r_c}+d) \overline{\Delta p} + f(t_c,r_c),
     \end{aligned}
    \]
   where in the equation the bar above means evaluation at $r_c$
   Then using again the estimates from \cite{cv11}, for small enough $\zeta$ we have a universal constant $M$ depending only on  $n,\sigma$ such that
    \[
     |\overline{\partial_r p}|, |\overline{\Delta p}| \leq Y_1(T)M.  
    \]
   Now recalling also that $LY_2(t_c)e^{-r_c}\leq Y_1(T)$
    \[
     \begin{aligned}
     \frac{2}{3}C &\leq \delta(2+d)Y_1(T) + MY_1(T) + \left(1+\frac{d}{LY_2(t_c)} e^{r_c} \right) MY_1(T) + \frac{f(r_c)e^{r_c}}{LY_2(t_c)} \\
                  &\leq \delta(2+d)Y_1(T) + 2MY_1{T}\left(1+\frac{d}{L} e^{r_c} \right) +\frac{A}{L}.
     \end{aligned}
    \]
   Choosing $\delta,d$ small enough the above inequality implies 
    \[
     C \leq 4MY_1(T) +4.
    \]
   If we choose now $C>2MY_1(T)+4$ we obtain a contradiction. We note that $C$ will only depend on 
   $n,\sigma, \|u(0,x) \|_{L^{\infty}},\|f \|_{L^{\infty}} $.
  \end{proof}

 We now give some Sobolev estimates. Because we have a right hand side we choose to not use $\ln(u)$ as the 
 test function. For $0<\gamma<1$, we use $(u+d)^\gamma-d^\gamma$ as a test function. 
  The function
   \[
    F(t)=\frac{1}{\g +1} (t+d)^{\g+1} -d^\g t
   \]
  will satisfy the conditions in Lemma \ref{l:discconvex}. 
  We now assume $u$ is a solution to \eqref{e:daprox} with assumptions as in Lemma \ref{l:tail}, so that 
  $|u|\leq Me^{-|x|}$ for some large $M$. 
  As discussed in the introduction we can extend $u(\e j,x)=u(0,x)$ for $j<0$, and $u$ will be a solution to 
  \eqref{e:daprox} on $(-\infty,T)\times \R^n$ with right hand side
   \[
    \delta \text{div}((u(0,x)+d)\nabla u(0,x)) + \text{div}((u(0,x)+d)\nabla (-\Delta)^{-\sigma}u(0,x)). 
   \]
  for $j \leq 0$. 
  This right hand side is not necessarily nonnegative; however, we only required the nonnegativity of the right hand side to guarantee that our
  solution is nonnegative. In this case we already know our solution is nonnegative. 
  We fix a smooth cut-off $\phi(t)$ with $\phi(t)\geq M$ for $t \leq -2$ and $\phi(t)=0$ for $t\geq -1$. 
   We
  now take our test function as $\e F'([u(t,x)-\phi(t)]_+)$. We define 
   \[
    u= (u-\phi)_+ - (u-\phi)_- + \phi =: u_{\phi}^+   - u_{\phi}^- + \phi. 
   \]
  We define $\tilde{u}=u(t)$ for $\e j-1<t\leq \e j$. From Lemma \ref{l:discconvex} 
  and the estimates in Section \ref{s:discretize}, there exist two constans $c,C$ depending on $\alpha,T, \Lambda$ such that for $\e <1$,
   \[
    \begin{aligned}
    &\e \sum_{j \leq k} F'(u_{\phi}^+(\e j))\D_{\e}^{\alpha} u(\e j) \\
     &\geq c\int_{-\infty}^T \int_{-\infty}^t \frac{[\tilde{u}_{\phi}^+(t) - \tilde{u}_{\phi}^+(s)]^2}{(t-s)^{1+\alpha}} \\
     &\quad + c\int_{-\infty}^T \frac{F(\tilde{u}_{\phi}^+(t))}{(T-t)^{\alpha}} \ dt \\
     &\quad - C\int_{-\infty}^T  F'(\tilde{u}_{\phi}^+(t)) D_t^{\alpha} \phi(t) \ dt
    \end{aligned}
   \]
  We now consider the nonlocal spatial term. 
  We will also use the following property: For an increasing function $V$ and a constant $l$
   \[
    \B_{\zeta}(V((u-l)_+),u) \geq \B_{\zeta}(V((u-l)_+),(u-l)_+) \geq 0.
   \]
  We have for the nonlocal spatial terms
   \[
    \begin{aligned}
     &\e \sum_{j\leq k} \int_{B_R} \nabla  F'(u_{\phi}^+(\e j,x)) (u+d) \nabla K_{\zeta} u \\
     &= \e \sum_{j\leq k} \int_{B_R} \nabla  F'(u_{\phi}^+(\e j,x)) [(u_{\phi}^+)+ d + \phi] \nabla K_{\zeta} u \\
     &= \frac{\g}{\g+1}\int_{-2}^{T} \B_{\zeta}((\tilde{u}_{\phi}^+ +d)^{\g+1}, u)  \\
     &\quad + \int_{-2}^{T} \phi(t)\B_{\zeta}((\tilde{u}_{\phi}^+ +d)^{\g}, u)  \\
     &\geq \frac{\g}{\g+1}\int_{-2}^{T} \B_{\zeta}((\tilde{u}_{\phi}^+ +d)^{\g+1}, u)  \\
     &\geq \frac{\g}{\g+1}\int_{-2}^{T} \B_{\zeta}((\tilde{u}_{\phi}^+ +d)^{\g+1}, u_{\phi}^+) 
    \end{aligned}
   \]
  From Proposition \ref{p:gamma}, if $u_{\phi}^+(x)-u_{\phi}^+(y)\geq0$, then
   \[
    (u_{\phi}^+ +d)^{\g+1}(x) - (u_{\phi}^+ +d)^{\g+1}(y) \geq (u_{\phi}^+(x) -u_{\phi}^+(y))^{\g+1}. 
   \]
  Then 
   \[
    \e \sum_{j\leq k} \int_{B_R} \nabla  F'(u_{\phi}^+(\e j,x)) (u+d) \nabla K_{\zeta} u 
    \geq \frac{c\g}{\g+1}\int_{0}^{T} \int_{\R^n} \int_{\R^n} H_{\zeta}(x,y)|\tilde{u}(x)-\tilde{u}(y)|^{2+\g} \ dx \ dy \ dt. 
   \]
  For the local spatial term we have
   \[
     \e \sum_{j\leq k} \int_{B_R} \nabla F'(u_{\phi}^+(\e j,x)) (u+d) \nabla u  
     \geq  \g \int_{0}^T \int_{B_R} (\tilde{u}+d)^{\g}|\nabla \tilde{u}|^2. 
   \]
  Now combining the previous estimates with the right hand side term $f$ we have for a certain constant $C$ depending on 
  $n,\sigma, \alpha, \Lambda, \gamma, M, T$ which can change line by line. 
   \[
    \begin{aligned}
     &\delta \g \int_{0}^T \int_{B_R} (\tilde{u}+d)^{\g}|\nabla \tilde{u}|^2 \\
     &\quad + \frac{c\g}{\g+1}\int_{0}^{T} \int_{\R^n} \int_{\R^n} H_{\zeta}(x,y)|\tilde{u}(x)-\tilde{u}(y)|^{2+\g} \ dx \ dy \ dt \\
     &\quad + c\int_{B_R}\int_{0}^T \int_{0}^t \frac{[\tilde{u}(t) - \tilde{u}(s)]^2}{(t-s)^{1+\alpha}} \\
     &\quad + c\int_{B_R}\int_{0}^T \frac{F(\tilde{u}(t))}{(T-t)^{\alpha}} \ dt \\
     &\leq C\int_{B_R} \int_{-2}^T F'(\tilde{u}_{\phi}^+(t)) D_t^{\alpha} \phi(t) \ dt \\
     &\quad + \int_{-2}^T \int_{B_R} f(t,x) F'(\tilde{u}_{\phi}^+) \\
     &\leq C \int_{-2}^T \int_{B_R} (\tilde{u}_{\phi}^+ +d)^{\g}-d^{\g} \\
     & \leq C \int_{-2}^T \int_{B_R} (Me^{-|x|}+d)^{\g}-d^{\g}
      \leq C \int_{-2}^T \int_{B_R} M^{\g} e^{-\g|x|} \leq C.
    \end{aligned} 
   \]
  The second to last inequality comes from Proposition \ref{p:exp}. The value $C$ is independent of 
  $\zeta, d, R, \e, \delta$ if 
   \[
    \zeta, \epsilon, \delta , d< 1 \text{ and }  R>1. 
   \]
  Then as $\zeta, d \to 0$ we have uniform control and obtain the estimate 
   \begin{equation}  \label{e:almost}
    \begin{aligned}
      &\delta \g \int_{0}^T \int_{B_R} (\tilde{u}+d)^{\g}|\nabla \tilde{u}|^2 
      + \int_{0}^T  \| \tilde{u}\|_{W^{(2-2\sigma)/(2+\g), 2+\g}(B_R)}^{2+\g}  \\
      &\quad + \int_{B_R} \|\tilde{u} \|_{W^{\alpha/2,2}(0,T)}^2  
     \leq C  
    \end{aligned}
   \end{equation} 
  Notice that the constant $C$ only depends on the exponential decay of $f, u_0$ and  on $\sigma, \alpha, n, T$, but not on 
  $R,\delta$. Letting $d, \zeta \to 0$ we obtain 
   \begin{equation}  \label{e:d2aprox}
    \D_{\e}^{\alpha} u - \delta \text{div}(u\nabla u) - \text{div}(u \nabla (-\Delta)^{-\sigma} u ) = f \text{  on  }  [0,T]\times B_R.
   \end{equation}
  
  We now give a compactness result.
  
  \begin{lemma}  \label{l:compact}
   Assume for any $v \in \mathcal{F}$,   
    \begin{equation}  \label{e:compact}
     \int_{0}^T  \| v(t,x)\|_{W^{(2-2\sigma)/(2+\g), 2+\g}(B_R)}^{2+\g}  
      + \int_{B_r} \|v(t,x) \|_{W^{\alpha/2,2}(0,T)}^2  
     \leq C.  
    \end{equation}
   Then $\mathcal{F}$ is totally bounded in $L^p([0,T]\times B_R)$ for $1\leq p\leq 2$. 
  \end{lemma}

  \begin{proof}
   We utilize the proof provided in \cite{dpv12} for compactness in fractional Sobolev spaces.  We will 
   show the result for $p=2$, and it will follow for $p<2$ since $B_R$ is a bounded set. We divide $T$ into $k$ increments. (This $k$ is unrelated to
   the number $k$ for the $\e$ approximations). Let $l = T/k$. We define
    \[
     v_l(x,t) := \frac{1}{l} \int_{jl}^{j(l+1)}v(x,s) \ ds. 
    \]
   From \cite{dpv12}
    \begin{equation}  \label{e:vl}
     \int_{B_R} \int_{0}^T [v_l(x,t)-v(x,t)]^2 \leq c_{\alpha} l^{\alpha} \int_{B_R} \|v(x,\cdot) \|_{W^{\alpha/2,2}(0,T)}^2 \leq Cl^{\alpha}.
    \end{equation}
   The above estimate is uniform for any $v_j$. We now utilize that $[0,T]$ is a finite measure space 
   as well as Minkowski's inequality: the norm
   of the sum is less than or equal to the sum of the norm.  
    \[
     \begin{aligned}
      C &\geq \sum_{j=0}^{k-1} \int_{lj}^{l(j+1)} \| v(\cdot, t) \|_{W^{(2-2\sigma)/(2+\g), 2+\g}(B_R)}^{2+\g} \\
        &\geq \sum_{j=0}^{k-1} \frac{1}{l^{1+\g}} \left( \int_{lj}^{l(j+1)} \| v(\cdot, t) \|_{W^{(2-2\sigma)/(2+\g), 2+\g}(B_R)} \right)^{2+\g} \\
        &\geq \sum_{j=0}^{k-1} \frac{1}{l^{1+\g}} \left( \left\| \int_{lj}^{l(j+1)}  v(x, t) \right\|_{W^{(2-2\sigma)/(2+\g), 2+\g}(B_R)} \right)^{2+\g} \\
        &= l \sum_{j=0}^{k-1} \left( \left\| \frac{1}{l}\int_{lj}^{l(j+1)}  v(x, t) \right\|_{W^{(2-2\sigma)/(2+\g), 2+\g}(B_R)} \right)^{2+\g} \\
     \end{aligned}
    \]
   It then follows from the result in \cite{dpv12} that for every $j$ and $\lambda>0$ there exists finitely many $\{\beta_1, \ldots, \beta_{M_j} \}$ such that 
   for any fixed $j$ and $v\in \mathcal{F}$ there exists $\beta_i \in \{\beta_1, \ldots, \beta_{M_j} \} $ such that 
    \[
     \int_{B_R} \left|\beta_i - \frac{1}{l}\int_{lj}^{l(j+1)}  v(x, t) \right|^2 \leq \lambda. 
    \]
   Then combining the above estimate with \eqref{e:vl} we obtain that 
    \[
     \begin{aligned}
     \int_{B_R} \int_{0}^T |v-\beta_{i,j}|^2 &\leq \int_{B_R} \int_{0}^T |v-v_l|^2+ \int_{B_R} \int_{0}^T |v_l-\beta_{i,j}|^2 \\
      &= \int_{B_R} \int_{0}^T |v-v_l|^2+  l\sum_{j=0}^{k-1}  \int_{B_R}|v_l-\beta_{i,j}|^2 \\
      & Cl^{\alpha} + T\lambda.
     \end{aligned}
    \]
   Since $l,\lambda$ can be chosen arbitrarily small, $\mathcal{F}$ is totally bounded. 
  \end{proof}
 
 The following result will guarantee that $\nabla (-\Delta)^{-\sigma} u \in L^p$ as $\delta \to 0, R \to \infty$. 
  \begin{lemma}  \label{l:triebel}
   Let $u$ be a solution to \eqref{e:d2aprox} with right hand side $f$ and $u_0$ both satisfying the exponential bound \eqref{e:exponent}.  
   Then 
    \[
     \int_{0}^T \| (-\Delta)^{-\sigma} u(t,\cdot)  \|_{W^{(2-2\sigma)/(2+\g) + 2\sigma, 2+\g}}^{2+\gamma} \ dt  \leq C 
    \]
   with the constant $C$ depending only on the exponential bounds in \eqref{e:exponent}, $n, \gamma,T$. 
  \end{lemma}
  
  \begin{proof}
   $u$ is extended to be zero outside of $B_R$. The proof is a consequence of the following results found in \cite{t10}. 
    \[
     W^{\beta,p}(\R^n) = B_{p,p}^{\beta}(\R^n) = L^p(\R^n) \cap \dot{B}_{p,p}^{\beta}.
    \]
   We also have the lifting property of the Riesz potential for the homogeneous Besov spaces
    \[
     \| (-\Delta)^{-\sigma} u\|_{\dot{B}_{p,p}^{\beta+2\sigma}} \leq C  \|  u\|_{\dot{B}_{p,p}^{\beta}}
    \]
   To bound $u$ in the nonhomogeneous Besov space we recall
    \[
     \| (-\Delta)^{-\sigma} u\|_{L^{nq/(n-2\sigma q)}(\R^n)} \leq C  \|  u\|_{L^q(\R^n)}.
    \]
   for any $1\leq q<n/(2\sigma)$. From the exponential bounds \eqref{e:exponent} and growth we have that $u$ is uniformly in $L^q$ for all $1\leq q \leq \infty$. 
   Letting
    \[
     q = \frac{(2+\g)n}{n+2\sigma(2+\g)} > 1 \text{  for  } \sigma < 1/2 \text{  and  }  n\geq2,
    \]
   (or let $\sigma <1/4$ for $n=1$), we obtain by the finite length of $T$
    \[
    \int_0^T \| (-\Delta)^{-\sigma} u\|_{L^{2+\g}(\R^n)}^{2+\g}\leq C.
    \] 
   Using again the characterization of homogeneous besov spaces we obtain the result. 
  \end{proof}

  \begin{corollary}
   Let $u_k$ be a sequence of solutions to \eqref{e:d2aprox} with $R \to \infty$ and $\delta \to 0$. For fixed $\rho>0$, there exists a subsequence and
   limit with 
    \[
     u_k \to u_0 \in L^p(B_{\rho}) \text{  for  } 1\leq p\leq 2 \text{  and   } u_k \rightharpoonup u_0 \in W^{(2-2\sigma)/(2+\g), 2+\gamma }.
    \]
   Furthermore, for any compactly supported $\phi$
    \begin{equation}  \label{e:recurse}
     \e \sum_{j\leq k} \int_{\R^n} \left[ \phi(x, \e j)\D_{\e}^{\alpha} u_0(\e j, x) + u_0\nabla \phi \nabla (-\Delta)^{-\sigma} u_0 \right]
      = \e \sum_{j\leq k}\int_{\R^n}f\phi
    \end{equation}
  \end{corollary}

  \begin{proof}
   The strong and weak convergence is an immediate result of the bound \eqref{e:almost} and Lemma \ref{l:compact}. 
   For $\gamma$ small enough depending on $\sigma$, then
    \[
     \frac{2-2\sigma}{2+\gamma} + 2\sigma >1. 
    \] 
   Then  
   from Lemma \ref{l:triebel} we have that 
    \[
     \nabla (-\Delta)^{-\sigma} u_k \rightharpoonup \nabla (-\Delta)^{-\sigma} u_0 \in W^{(2-2\sigma)/(2+\g)+2\sigma-1, 2+\gamma },
    \]
   And in particular
    \begin{equation}  \label{e:pconv}
     \nabla (-\Delta)^{-\sigma} u_k \rightharpoonup \nabla (-\Delta)^{-\sigma} u_0 \in L^{2+\g}(\R^n).
    \end{equation}
   Then it is immediate from the weak and strong convergence that $u_0$ is a solution. 
  \end{proof}

 
  We now show the 
  \begin{proof}[Proof of Theorem \ref{t:existence}]
   We first assume $f, u_0$ smooth and satisfying the exponential bounds \eqref{e:exponent}. 
   Consider solutions $u_{\e}$ to \eqref{e:recurse} over a finite interval $(0,T)$.   
   As before, as $\e \to 0$ there exists a subsequence and a limit 
   $u_{\e} \to u_0$ with the weak convergence as in \eqref{e:pconv} and strong convergence over compact sets for $1\leq p \leq 2$ just as in Lemma
   \ref{l:compact}. Then for fixed $\phi\in C_0^{\infty}$, that $u_0$ is a solution follows from this convergence. The spatial piece and right hand
   side is straightforward to show, and the nonlocal time piece is taken care of as in \cite{acv15}. We now consider a sequence of 
   solution $\{u_j\}$ with  $\{f_j\}, \{(u_0)_j\} \in C^{\infty}$
   with $f_j \to f$ and $ (u_0)_j \to u_0$ in weak$^*$ $L^{\infty}$. Then again there exists a limit solution $u$ 
   with right hand side $f$. From Remark \ref{r:unique}
   and Lemma \ref{l:boundab} we can let $T \to \infty$. 
  \end{proof}

 \begin{remark}  \label{r:finite}
  In this Section we have shown how the estimates in \cite{cv11} work for equations of the form \eqref{e:main}. In the same way one can show that 
  the method of ``true (exaggerated)  supersolutions'' as shown in \cite{cv11} for $\sigma <1/2$ will also work to prove the property 
  of finite propagation for solutions
  to \eqref{e:main}. As the main result of this paper is H\"older regularity of solutions we will not make this presentation here.  
 \end{remark}

\section{Continuity: Method and Lemmas}  \label{s:lemmas}

 In this Section we outline the method used to prove H\"older regularity of solutions to \eqref{e:main}. We follow the method used in \cite{cfv13} which
 is an adaption of the ideas originally used by De Giorgi. We prove a decrease in oscillation on smaller cylinders and then 
 utilize the scaling property that if $u$ is a solution to \eqref{e:main}
 , then $v(t,x)=A(Bt,Cx)$ is also a solution to \eqref{e:main} if $A=B^{\alpha}C^{2-2\sigma}$. 
 Because of the degenerate nature of the problem the decrease in oscillation will only occur from above. Since we do not have a decrease in oscillation from 
 below we will need a Lemma that says in essence that if the solution $u$ is above $1/2$ on most of the space time, then $u$ is a distance from zero on 
 a smaller cylinder. To prove the Lemmas in this section we will use energy methods, and thus we will want to use as a test function $F(u)$ for some $F$. 
 If $u$ is a solution to \eqref{e:main}, then 
  \[
   u \in W^{(2-2\sigma)/(2+\g),2+\g}
  \]
 and it is not clear that $\nabla F(u)$ will be a valid test function. We therefore prove the Lemmas for the approximate problems 
  \begin{equation}  \label{e:1delaprox}
   \D_t^{\alpha} u - \delta \text{div}(D(u) \nabla u) - \text{div}(u\nabla (-\Delta)^{-\sigma} u) =f \text{ on } B_R,
  \end{equation}
 for some large $R>0$ and small $\delta>0$ with $u \equiv 0$ on $\partial B_R$. It is actually only necessary to prove the energy inequalities that 
 we will utilize with constants uniform as $\delta \to 0$ and $R \to \infty$.  
 We could also prove the Lemmas for the approximate problems \eqref{e:d2aprox}; however, for notational convenience and to make the proofs more
 transparent we have chosen to let $\e \to 0$. 
 Because our
 solution is a limit of discretized solutions we then are allowed to make the formal computations involved 
  with $\D_t^{\alpha} u$
 even though $u$ may not be regular enough for $\D_t^{\alpha} u$ to be defined. One simply proves the energy 
 inequalities (and hence the Lemmas) for the discretized solutions as was done in \cite{acv15}.

 %
 
 Because of the one-sided nature of our problem we prove the Lemmas for solutions to the equation with the modified term 
 div$(D(u)\nabla (-\Delta)^{-\sigma}u)$, where $D(u)=d_1u +d_2$. We assume $0\leq d_1, d_2 \leq 2$ and either $d_1=1$ or $d_2 \geq 1/2$. 
 As will be seen later, when $d_2 \geq 1/2$, the proofs are simpler because the problem is no longer
 degenerate.  
 We now define the exact class of solutions for which we prove the Lemmas of this section. 
 $u$ is a solution if $u \equiv 0$ on $\partial B_r$ and for every $\phi \in C_0^{\infty}((-\infty,T)\times B_R) $, we have 
  \begin{equation}  \label{e:2delaprox}
    \begin{aligned}
     &\int_{B_R} \int_{-\infty}^T \int_{-\infty}^t [u(t,x)-u(s,x)][\phi(t,x)-\phi(s,x)]K(t,s) \ ds \ dt \ dx \\
     &\quad + \int_{B_R} \int_{-\infty}^T \int_{-\infty}^{2t-T} u(t,x)\phi(t,x) K(t,s) \ ds \ dt \ dx \\
     &\quad -\int_{B_R} \int_{-\infty}^T u(t,x) \D_t^{\alpha} \phi(t,x) \ dt \ dx \\
     &\quad +\int_{-\infty}^T \int_{B_R} \nabla \phi(t,x) D(u)\nabla u(t,x) \ dx \ dt \\
     &\quad +\int_{-\infty}^T \int_{B_R} \nabla \phi(t,x) D(u) \nabla (-\Delta)^{-\sigma} u \ dx \ dt\\
     &= \int_{-\infty}^T \int_{B_R} f(t,x) \phi(t,x),
    \end{aligned} 
  \end{equation}
 By Lemmas \ref{l:compact} and \ref{l:triebel}, the Lemmas stated in this section will be true when $R \to \infty$ and $\delta \to 0$.

  Before stating the Lemmas we define the following function for small $0<\tau <1/4$. 
   \[
    \overline{\Psi}(x,t) := 1+ (|x|^{\tau} -2)_+ + (|t|^{\tau} -2)_+.
   \]
  We now state the Lemmas we will need.
  
  \begin{lemma}  \label{l:pullup}
   Let $u$ be a solution to \eqref{e:2delaprox} with $R>4$ and assume 
    \[
      1-\overline{\Psi}\leq u \leq \overline{\Psi} \text{  for  }  \tau < \tau_0
    \]
   Given $\mu_0 \in (0,1/2)$ and $\tau_0<1/4$, there exists $\kappa>0$
   depending on $\mu_0, \tau_0, \sigma, \alpha, n$ such that if 
    \[
      |\{u \geq 1/2\} \cap \Gamma_4| \geq (1-\kappa) |\Gamma_4|
    \]
   then $u \geq \mu_0$ on the smaller cylinder $\Gamma_1$. 
  \end{lemma}

  We have a similar Lemma from above 
   \begin{lemma}  \label{l:pulldown}
   Under the same assumptions as Lemma \ref{l:pullup}, given $\mu_1 \in (0,1/2)$ and $\tau_0<1/4$, there exists $\kappa>0$
   depending on $\mu_1, \tau_0, \sigma, \alpha, n$ such that if 
    \[
      |\{u > 1/2\} \cap \Gamma_2| \leq \kappa |\Gamma_2|
    \]
   then $u \leq 1-\mu_1$ on the smaller cylinder $\Gamma_1$. 
  \end{lemma}

  Lemma \ref{l:pulldown} is not sufficient. We need the stronger
   \begin{lemma}  \label{l:2down}
   Under the same assumptions as Lemma \ref{l:pullup}, assume further for fixed $k_0$
    \begin{equation} \label{e:2down}
     |\{u<1/2\}\cap \Gamma_4| \geq (1-\kappa_0)|\Gamma_4|,
    \end{equation}
   then $u \leq 1-\mu_2$ on $\Gamma_1$ for some $\mu_2$ depending on $\kappa_0$. 
  \end{lemma}
  We will choose  $\kappa_0$ to equal the $\kappa$ in Lemma \ref{l:pullup}.

\section{Pull-up}   \label{s:pullup}
 In this section we provide the proof of Lemma \ref{l:pullup}. This Lemma is the most technical to prove. We first prove the Lemma in
 the most difficult case when $D(u)=u+d$ with $0\leq d \leq 2$. Afterwards, we show how the proof is much simpler when $D(u)=d_1 u +d_2$ with 
 $d_2 \geq 1/2$ and $0\leq d_1 \leq 2$.

 We will need the following technical Lemma. The proof is found in the appendix.  
  \begin{lemma}  \label{l:control}
   Let $u,\phi$ be two functions such that $0\leq u \leq \phi \leq 1$. Let $0<\gamma<1$ be a constant. 
   If $|u(x)-u(y)|\geq 4|\phi(x)-\phi(y)|$, then 
    \begin{equation} \label{e:control1}
     \frac{2}{5}\left(\frac{4}{5} \right)^{\g} |u_{\phi}^-(x)-u_{\phi}^-(y)|^{1+\g} 
      \leq \left| \frac{u^{\g+1}(y)}{\phi^{\g}(y)} - \frac{u^{\g+1}(x)}{\phi^{\g}(x)} \right|
      \leq \frac{14}{3} |u_{\phi}^-(x)-u_{\phi}^-(y)|.
    \end{equation}
   Also, if  
    \[
     0 \leq  \frac{u^{\g+1}(y)}{\phi^{\g}(y)} - \frac{u^{\g+1}(x)}{\phi^{\g}(x)} 
    \] 
   then $0\leq u_{\phi}^-(x)-u_{\phi}^-(y)$. 
   
   If instead we assume $|u(x)-u(y)|\leq 4|\phi(x)-\phi(y)|$, then 
    \begin{equation} \label{e:control2}
      \left| \frac{u^{\g+1}(y)}{\phi^{\g}(y)} - \frac{u^{\g+1}(x)}{\phi^{\g}(x)} \right|
      \leq 14 |\phi(x)-\phi(y)|.
    \end{equation}
  \end{lemma}
  
 \begin{remark}  \label{r:2}
  When $0\leq u \leq \phi \leq 3$,  Lemma \ref{l:control} will hold with new constants by applying the Lemma to $u/3,\phi/3 $.  
 \end{remark}

 We will use a sequence of cut-off functions $\{\phi_k\}$ which will be chosen to be smooth cut-off functions in space, 
 and  smooth increasing cut-off functions
 in time. 
  We recall that for small $0<\tau <1/2$, 
   \[
    \overline{\Psi}(x,t) := 1+ (|x|^{\tau} -2)_+ + (|t|^{\tau} -2)_+.
   \]
 We now recall the construction of a sequence of smooth radial cut-offs $\theta_k$ from \cite{cfv13} that satisfy 
 \begin{itemize}
  \item $\theta_k(x)\leq \theta_{k-1}(x)\leq \ldots \leq \theta_0(x) $,
  \item $|\nabla \theta_k|/\theta_k \leq C^k \theta_k^{-1/m}$, with $m \geq 2$
  \item $\theta_{k-1}-\theta_k\geq (1-\mu_0)2^{-k}$ in the support of $\theta_k$,
  \item $\theta_k \to \mu_0 \chi_{B_2}$ as $k \to \infty$,
  \item the support of $\theta_k$ is contained in the set where $\theta_{k-1}$ achieves its maximum.  
 \end{itemize}
 We also have 
  $\theta_0 \equiv 1$ on $B_3$  and the support of $\theta_0$ is contained in $B_4$. 
 As a cut-off in time we consider a sequence $\{\xi_k\}$ satisfying 
  \begin{itemize}
  \item $\xi_k(t) \leq \xi_{k-1}(t)$,
  \item $\xi_k'(t) \leq C^k $,
  \item $\xi_k \to \chi_{\{t>-2\}}$ as $k \to \infty$,
  \item $\xi_k =\max \xi_k = 1$ on the interval $[-2-2^{-k},0]$. 
  \item the support of $\xi_k$ is contained in the set where $\xi_{k-1}$ achieves its maximum.  
 \end{itemize}
 We now define 
  \[
   \phi_k(x,t) := 1-\overline{\Psi}(x,t) + \frac{1}{2}\xi_k(t)\theta_k(x). 
  \]
 We use the convention for negative part that $u=u_+ - u_-$. We also write $u_{\phi_k}^- := (u-\phi_k)_-$.
 We now consider the convex function
  \begin{equation}  \label{e:ftest}
   F(x):= \frac{1}{\g + 1}(1-x)^{\g +1} +x -\frac{1}{\g +1}.
  \end{equation}
 Because of the degenerate nature of our equation we will want to utilize the test function
  \begin{equation} \label{e:test}
    -F'(u_{\phi}^-/(\phi+d))=  \left(1- \frac{(u-\phi)_-}{\phi+d} \right)^{\g} -1 
     = -\left[\left(\frac{u+d}{\phi+d}\right)^{\g}-1\right]_-.
   \end{equation}

 \begin{proof}[Proof of Lemma \ref{l:pullup}]
 
 \textbf{First Step: Obtaining an energy in time.}
 We note that for $0\leq x \leq 1$, $F(x)$ is convex,$F'(x)\geq 0$, and $F''(x)\geq \g$.  From the convexity and 
 second derivative estimate we also conclude for $0\leq x,y\leq 1$
  \begin{gather}   \label{e:convexity}
     F'(x)(x-y) \geq F(x)-F(y) + (\g/2) (x-y)^2 \\
     \label{e:xsquared}F(x) \approx  x^2
  \end{gather}
 We now consider $- F'(u_{\phi_k}^-/(\phi_k+d)) \D_t^{\alpha} u$, and rewrite $u = u_{\phi_k}^+ - u_{\phi_k}^- + \phi_k$. 
 To obtain an energy in time we first consider 
  \[
   \begin{aligned}
    &\int_{-\infty}^0 F' \left(\frac{u_{\phi}^-(t)}{\phi(t)+d} \right) \D_t^{\alpha} u_{\phi}^- (t) \\
    &= \int_{-\infty}^0 \int_{-\infty}^0 F' \left(\frac{u_{\phi}^-(t)}{\phi(t)+d} \right) \left[u_{\phi}^-(t)-u_{\phi}^-(s)\right] K(t,s) \\
    &=  \int_{-\infty}^0 \int_{-\infty}^0(\phi(t)+d) F' \left(\frac{u_{\phi}^-(t) }{\phi(t)+d} \right) 
            \frac{u_{\phi}^-(t)-u_{\phi}^-(s)}{\phi(t)+d} K(t,s) \\
    &\geq \int_{-\infty}^0 \int_{-\infty}^0 (\phi(t)+d)F' \left(\frac{u_{\phi}^-(t)}{\phi(t)+d} \right) 
      \left[ \frac{u_{\phi}^-(t)}{\phi(t)+d}-\frac{u_{\phi}^-(s)}{\phi(s)+d} \right]K(t,s) \\
    &\geq  \int_{-\infty}^0 \int_{-\infty}^0 (\phi(t)+d)\left[F\left(\frac{u_{\phi}^-(t)}{\phi(t)+d} \right) 
                            - F\left(\frac{u_{\phi}^-(s)}{\phi(s)+d} \right)\right] K(t,s) \\
    & \  + \int_{-\infty}^0 \int_{-\infty}^0 (\phi(t)+d) 
       \frac{\gamma}{2} \left[\frac{u_{\phi}^- (t)}{\phi(t)+d} - \frac{u_{\phi}^-(s)}{\phi(s)+d} \right]^2 K(t,s)\\
    &\geq \int_{-\infty}^0 \int_{-\infty}^0 (\phi(t)+d)\left[F\left(\frac{u_{\phi}^-(t)}{\phi(t)+d} \right) 
                            - F\left(\frac{u_{\phi}^-(s)}{\phi(s)+d} \right)\right] K(t,s) \\
    & \ + \int_{-2-2^{-k}}^0 \int_{-2-2^{-k}}^0  
       \frac{\gamma}{2} \left[u_{\phi}^- (t) - u_{\phi}^-(s)\right]^2 K(t,s)\\
    &= (1)+(2).   
   \end{aligned}
  \]
 In the first inequality we used that $\phi$ is increasing in $t$ and positive for $t\geq -4$ as well as $u_{\phi}^-(s)=0$ for $s\leq -4$, 
 and in the second inequality we used \eqref{e:convexity}. Term $(2)$ is half of what we will need for the Sobolev embedding 
 (see Lemma \ref{l:ext}).  To gain
 the other half we consider term $(1)$. 
 For $c,C^k$ depending on $\Lambda,\alpha$ and the Lipschitz constant
 of $\phi_k$ we have
  \[
   \begin{aligned}
    & \int_{-\infty}^0 \int_{-\infty}^0(\phi_k(t)+d)\left[F\left(\frac{u_{\phi_k}^-(t)}{\phi_k(t)+d} \right) 
          - F\left(\frac{u_{\phi_k}^-(s)}{\phi_k(s)+d} \right)\right] K(t,s) \\
     &= \int_{-\infty}^0 \int_{-\infty}^0 \left[(\phi_k(t)+d)F\left(\frac{u_{\phi_k}^-(t)}{\phi_k(t)+d} \right) 
          - (\phi_k(s)+d)F\left(\frac{u_{\phi_k}^-(s)}{\phi_k(s)+d} \right)\right] K(t,s) \\
     &\quad + \int_{-\infty}^0 \int_{-\infty}^0 [\phi_k(s)-\phi_k(t)]F\left(\frac{u_{\phi_k}^-(s)}{\phi_k(t)+d} \right) K(t,s) \\
     &\geq c \int_{-\infty}^0 (\phi(t)+d) F\left(\frac{u_{\phi_k}^-(t)}{\phi_k(t)+d} \right) \frac{1}{(0-t)^{\alpha}} \\
     &\quad - C^k \int_{-\infty}^0 \chi_{\{u(t)<\phi_k(t)\}}  \ dt \\
     &\geq  c \int_{-2-2^{-k}}^0 \frac{(u_{\phi_k}^-)^2(t)}{(\phi_k(t)+d)(0-t)^{\alpha}} \\
     &\quad - C^k \int_{-\infty}^T \chi_{\{u(t)<\phi_k(t)\}}  \ dt \\
   \end{aligned}
  \] 
 The last inequality coming from \eqref{e:xsquared}.
  
 Now 
  \[
   -\int_{-\infty}^0 F' \left(\frac{u_{\phi}^-(t)}{\phi(t)+d} \right) \D_t^{\alpha} u_{\phi}^+ (t) \geq 0,
  \]
 and in this proof we ignore this term which will be on the left hand side. Now for the term involving $\phi_k$ we have 
  \[
   -\int_{-\infty}^0 F' \left(\frac{u_{\phi_k}^-(t)}{\phi_k(t)+d} \right) \D_t^{\alpha} \phi_k (t) 
    \geq -C^k \int_{-\infty}^0 \chi_{\{u<\phi_k\}}.
  \]
 Then utilizing the embedding theorem for fractional Sobolev spaces \cite{dpv12} 
 combined with the above inequalities we obtain 
  \begin{equation}  \label{e:timeenergy}
   \begin{aligned}
   \int_{-\infty}^0 F'\left(\frac{u_{\phi_k}^-}{\phi_k+d} \right) \D_t^{\alpha}u(t)  
   &\geq c \int_{-2-2^{-k}}^0 \frac{(u_{\phi_k}^-)^2(t)}{(0-t)^{\alpha}} \\
   & \quad + c \int_{-2-2^{-k}}^0 \int_{-2-2^{-k}}^t  \frac{\left[u_{\phi_k}^-(t) - u_{\phi_k}^-(s)\right]^2}{(t-s)^{1+\alpha}} \\
   & \quad - C^k \int_{-\infty}^0 \chi_{\{u<\phi_k\}} \ dt \\
   &\geq c \left( \int_{-2-2^{-k}}^0 \left(u_{\phi_k}^-\right)^{\frac{2}{1-\alpha}}\right)^{1-\alpha} - C^k \int_{-\infty}^0 \chi_{\{u<\phi_k\}} \ dt.
   \end{aligned}
  \end{equation}
 

 After integrating in the spatial variable we have
  \[
   \begin{aligned}
    &c\int_{\R^n} \left( \int_{-2-2^{-k}}^T (u-\phi_k)_-^{\frac{2}{1-\alpha}}\right)^{1-\alpha}
     - \int_{-\infty}^0 \int_{\R^n} (u+d) \nabla F'\left(\frac{u_{\phi_k}^-}{\phi_k+d} \right) \nabla (-\Delta)^{-\sigma} u \\
    &\quad - \int_{-\infty}^0 \int_{\R^n} (u+d) \nabla F'\left(\frac{u_{\phi_k}^-}{\phi_k+d} \right) \nabla  u \\
    &\leq  -\int_{-\infty}^0 \int_{\R^n} f F'\left(\frac{u_{\phi_k}^-}{\phi_k+d} \right) + C^k\int_{\R^n} \int_{-\infty}^0 \chi_{\{u<\phi_k\}} \ dt \ dx \\
    &\leq C^k\int_{-\infty}^0 \int_{\R^n}  \chi_{\{u<\phi_k\}} \ dt \ dx \\
   \end{aligned}
  \]

 \textbf{Second Step: Obtaining an energy in space.}
 We now turn our attention to the elliptic portion of the problem. 
 We recall from \cite{cfv13} the identity 
  \[
   \B(v,w) = \int_{\R^n}\int_{\R^n} \frac{[v(x)-v(y)][w(x)-w(y)]}{|x-y|^{n+2-2\sigma}} dx dy = c_{n,\sigma} 
     \int_{\R^n} \nabla v \nabla (-\Delta)^{-\sigma}u . 
  \] 
 We multiply by our test function \eqref{e:test} and integrate by parts. On the left hand side of the equation we have
  \[
   \begin{aligned}
    &\chi_{\{u<\phi\}}\nabla\left[\left(\frac{u+d}{\phi+d}\right)^{\g} -1\right] (u+d) \nabla (-\Delta)^{-\sigma}u \\
    &= \chi_{\{u<\phi\}}\g \left(\frac{u+d}{\phi+d}\right)^{\g-1} \nabla((u+d)/(\phi+d)) (u+d) \nabla (-\Delta)^{-\sigma} u \\
    &= \chi_{\{u<\phi\}}\frac{\g}{\g+1} \nabla \left[\left(\frac{u+d}{\phi+d}\right)^{\g+1} -1\right](\phi+d) \nabla (-\Delta)^{-\sigma} u \\
    &= \chi_{\{u<\phi\}}\frac{\g}{\g+1} \nabla \left[\frac{(u+d)^{\g+1}}{(\phi+d)^{\g}}-(\phi+d) \right]  \nabla (-\Delta)^{-\sigma} u \\
    &\qquad   -\frac{\g}{\g+1} \chi_{\{u<\phi\}} \left[\left({\frac{u+d}{\phi+d}}\right)^{\g+1}-1\right]\nabla \phi  \nabla (-\Delta)^{-\sigma} u\\
    &:= (1)+(2).   
   \end{aligned}
  \]
 We now focus on $(1)$ which will give us the energy term we need. 
 For the term $(-\Delta)^{-\sigma} u$,
 we rewrite $u=(u-\phi_k)_+ - (u-\phi_k)_- + \phi_k := u_{\phi_k}^+ - u_{\phi_k}^- + \phi_k$. 
 Then we rewrite $(1)=(1a)+(1b)+(1c)$.
 We focus on the term $(1b)$. 
 We rewrite 
  \[
   \begin{aligned}
   (1b)&=(1bi)+(1bii) \\
       &:= -\chi_{\{u<\phi\}}\frac{\g}{\g+1} \nabla \left[\frac{(u+d)^{\g+1}}{(\phi+d)^{\g}} \right]  \nabla (-\Delta)^{-\sigma} u_{\phi}^- \\
       &\quad + \chi_{\{u<\phi\}}\frac{\g}{\g+1} \nabla \phi \nabla (-\Delta)^{-\sigma} u_{\phi}^-.
   \end{aligned}
  \] 
 The term $(1bi)$ will give us the energy term in space that we will need. 
  \[
   \begin{aligned}
   (1bi)&=\int_{\R^n}\int_{\R^n} \frac{\g}{\g+1} 
   \nabla \left[\frac{(u+d)^{\g+1}}{(\phi+d)^{\g}}(x) \right] \frac{1}{|x-y|^{n-2\sigma}} \nabla u_{\phi}^-(y) dx dy \\
    &= c_{n,\sigma}\frac{\g}{\g+1} \B(\chi_{\{u<\phi\}}(u+d)^{\g+1} / (\phi+d)^{\g}, -u_{\phi}^-). 
   \end{aligned}
  \]
  We define the set 
  \[
   A_k:= \{|u(x)-u(y)| \geq 4 |\phi_k(x)-\phi_k(y)|\}.
  \]
 It is clear that $A_k$ contains the set $V_k \times V_k$ where we define $V_k$ as the set on 
 which $\theta_k$ achieves its maximum. 
 From Lemma \ref{l:control} and Remark \ref{r:2} we have 
  \[
   \iint\limits_{A_k} \left[ \frac{(u+d)^{\g+1}(y)}{(\phi+d)^{\g}(y)} - \frac{(u+d)^{\g+1}(x)}{(\phi+d)^{\g}(x)}\right]
    \frac{[u_{\phi_k}^-(x)-u_{\phi_k}^-(y)]}{|x-y|^{n+2-2\sigma}} 
   \geq c \iint\limits_{A_k} \frac{|u_{\phi_k}^-(x)-u_{\phi_k}^-(y)|^{2+\g}}{|x-y|^{n+2-2\sigma}}
  \]
 We now label $U_k$ as the set where $\phi_k$ achieves its maximum. Notice that $U_k = [-2-2^{-K}, 0] \times V_k$. 
 To utilize the fractional Sobolev embedding on $V_k \times V_k$, we also will need an $L^p$ norm of $u_{\phi_k}^-$ on $V_k$. We utilize half of the 
 integral of $u_{\phi_k}^-$ that we gained from the fractional time term:
  \[
   \begin{aligned}
   & \int_{\R^n} \int_{-2-2^{-k}}^0 \frac{(u_{\phi_k}^-)^2(t)}{(0-t)^{\alpha}} \\
    &\geq \frac{1}{2} \iint\limits_{U_k} \frac{(u_{\phi_k}^-)^2(t)}{(0-t)^{\alpha}} 
     + \frac{1}{2} \iint\limits_{U_k} \frac{(u_{\phi_k}^-)^{2+\g}(t)}{(0-t)^{\alpha}}.
   \end{aligned}
  \]
 The inequality comes from the fact that $0\leq u_{\phi}^-\leq 1$.  
 Now from the fractional sobolev embedding \cite{dpv12},
  \begin{equation}  \label{e:leftene}
   \begin{aligned}
    & \int_{-2-2^{-k}}^T \iint\limits_{V_k \times V_k} \frac{|u_{\phi_k}^-(x)-u_{\phi_k}^-(y)|^{2+\g}}{|x-y|^{n+2-2\sigma}}
     +  \frac{1}{2} \iint\limits_{U_k} (u-\phi_k)_-^{2+\g}(t) \\
    &\geq c_{n,\sigma, \gamma} \int_{-2-2^{-k}}^T \left( \int_{V_k} (u_{\phi_k}^-)^{n(2+\g)/(n-2+2\sigma)} \right)^{(n-2+2\sigma)/n}  
   \end{aligned}  
  \end{equation}
 This is the helpful spatial term on the left hand side that we will return to later.

 \textbf{Third Step: Bounding the remaining terms.}
 We will now show that everything left in our equation can be bounded by 
  \begin{equation}  \label{e:rhs}
   C^k \int_{-\infty}^T\int_{\R^n} \chi_{\{u < \phi_k\}} \ dx. 
  \end{equation}
  We will denote 
  \[
   X_k(x,y):= \chi_{\{u(x)<\phi_k(x)\}} +  \chi_{\{u(y)<\phi_k(y)\}}
  \]
 For the remainder of term $(1bi)$ we have 
  \[
   \begin{aligned}
    &\left |\iint\limits_{A_k^c} \left[ \frac{(u+d)^{\g+1}(y)}{(\phi+d)^{\g}(y)} - \frac{(u+d)^{\g+1}(x)}{(\phi+d)^{\g}(x)}\right]
    \frac{[u_{\phi_k}^-(x)-u_{\phi_k}^-(y)]}{|x-y|^{n+2-2\sigma}} \right| \\
    &\leq C \iint\limits_{A_k^c} X_k(x,y)\frac{|\phi_k(x)-\phi_k(y)|^{2}}{|x-y|^{n+2-2\sigma}} \\
    &\leq C^k \int_{\R^n} \chi_{\{u <\phi_k\}}. 
   \end{aligned}
  \]
 The last inequality is due to the Lipschitz constant of $\phi_k$ when $x,y$ are close, and the tail growth of $\phi_k$ when $x,y$
 are far apart.

 We now control the term $(1bii)$. 
 Again, we split the region of integration over $A_k$ and $A_k^c$. Using H\"older's inequality (provided $2\sigma>\g/(1+\g)$ and therefore
 we must choose $\gamma$ small when $\sigma$ is small) 
 as well as the Lipschitz and $\sup$ bounds on $\phi_k$ we have
  \[
   \begin{aligned}
   (1bii) &= c_{n,\sigma}\iint\limits_{A_k} \frac{[\phi_k(x)-\phi_k(y)][u_{\phi_k}^-(x)-u_{\phi_k}^-(y)]}{|x-y|^{n+2-2\sigma}} dx dy. \\
          &\leq \eta \iint\limits_{A_k} \frac{[u_{\phi_k}^-(x)-u_{\phi_k}^-(y)]^{2+\g}}{|x-y|^{n+2-2\sigma}} dx dy. \\
          &\quad + C \iint\limits_{A_k} \frac{[\phi_k(x)-\phi_k(y)]^{(2+\g)/(1+\g)}}{|x-y|^{n+2-2\sigma}} X_k(x,y)dx dy. \\
          &\leq \eta \iint\limits_{A_k} \frac{[u_{\phi_k}^-(x)-u_{\phi_k}^-(y)]^{2+\g}}{|x-y|^{n+2-2\sigma}} dx dy. \\
          &\quad + C^k \int_{\R^n}\chi_{\{u<\phi_k\}} \ dx.    
   \end{aligned}
  \]
  The first term is absorbed into the left hand side and the second term is controlled exactly as before.

 We now consider the integration over $A_k^c$. 
  \[
   \begin{aligned}
   (1bii) &= \iint\limits_{A_k} \frac{[\phi_k(x)-\phi_k(y)][u_{\phi_k}^-(x)-u_{\phi_k}^-(y)]}{|x-y|^{n+2-2\sigma}} dx dy. \\
          &\leq  \iint\limits_{A_k} \frac{[\phi_k(x)-\phi_k(y)]^{2}}{|x-y|^{n+2-2\sigma}} X_k(x,y)dx dy. \\
          &\leq  C^k \int_{\R^n}\chi_{\{u<\phi_k\}} \ dx.    
   \end{aligned}
  \]


 We now turn our attention to the term $(1c)$. By Lemma \ref{l:control} we have 
  \[
   \begin{aligned}
    &\left| \B(\chi_{\{u<\phi\}} ((u+d)^{\g+1}/(\phi_k+d)^{\g}-\phi_k),\phi_k) \right| \\
    &\leq \left| \B(\chi_{\{u<\phi\}} (u+d)^{\g+1}/(\phi_k+d)^{\g},\phi_k) \right| 
       + \left| \B(\chi_{\{u<\phi\}} \phi_k),\phi_k) \right| \\
   \end{aligned}
  \]
 Both of the above terms are handled exactly as before by using Lemma \ref{l:control} and splitting the region of integration over 
 $A_k$ and $A_k^c$. 

 The term $(1a)$ is  
  \begin{equation}  \label{e:posi}
   \begin{aligned}
   (1a)&=\B(\chi_{\{u<\phi_k\}} u^{\g+1}/\phi_k^{\g}-(\phi_k+d),u_{\phi_k}^+) \\
       &= 2\int_{\R^n}\int_{\R^n} \chi_{\{u(x)<\phi_k(x)\}}\left[\phi_k(x)+d-\frac{(u+d)^{\g+1}(x)}{(\phi_k+d)^{\g}(x)} \right]  
    \frac{u_{\phi_k}^+(y)}{|x-y|^{n+2-2\sigma}} \ dx \ dy \geq 0
   \end{aligned}
  \end{equation}
 The factor of $2$ comes form the symmetry of the kernel. We will utilize this nonnegative term shortly. 
 
 We now consider the term $(2)$ which we recall as
  \[
   -\int_{\R^n}\int_{\R^n}\frac{\g}{\g+1} \chi_{\{u<\phi_k\}} (((u+d)/(\phi_k+d))^{\g+1}-1)\nabla \phi_k  \nabla L(x-y)[u(y)-u(x)] \ dx \ dy.
  \]
 In the above $L = \nabla (-\Delta)^{-\sigma}$ and 
 we have
  \[
   \nabla L(x-y) \approx |x-y|^{-(n+1-2\sigma)}. 
  \]
 We again write $u=u_{\phi_k}^+ -u_{\phi_k}^-  + \phi_k$. 
 To control the term involving $\phi_k$ we integrate over the two sets $\{|x-y|\leq 8\}$ and $\{|x-y|>8\}$. 
 We use that $|\phi_k(x)-\phi_k(y)|\leq C^k|x-y|$ when $|x-y|\leq 8$ and $|\phi_k(x)-\phi_k(y)|\leq |x-y|^\tau$ when $|x-y|>8$ as well as the bound 
 $|\nabla \phi_k|\leq C^k$ to obtain 
  \[
   \begin{aligned}
    &\left| \int_{\R^n}\int_{\R^n}\frac{\g}{\g+1} \chi_{\{u<\phi_k\}} \left(\left(\frac{u+d}{\phi_k+d}\right)^{\g+1}-1\right)
    \nabla \phi_k  \nabla L(x-y)[\phi_k(y)-\phi_k(x)] \ dx \ dy \right| \\
    &\leq C^k\iint\limits_{|x-y|\leq 8} \chi_{\{u<\phi_k\}}  |x-y|^{-(n-2\sigma)} \ dx \ dy \\
    & \quad + C^k \iint\limits_{|x-y| > 8} \chi_{\{u<\phi_k\}} |x-y|^{-(n+1-2\sigma-\tau)}\ dx \ dy \\
    &\leq C^k \int_{\R^n} \chi_{\{u<\phi_k\}} \ dx. 
   \end{aligned}
  \]
 We now use the same set decomposition with $-u_{\phi_k}^-$, the inequality $|u_{\phi_k}^-|\leq 1$ 
 as well as H\"older's inequality 
  \[
   \begin{aligned}
    &\left| \int_{\R^n}\int_{\R^n}\frac{\g}{\g+1} \chi_{\{u<\phi_k\}} 
        \left(\left(\frac{u+d}{\phi_k+d}\right)^{\g+1}-1\right)\nabla \phi_k  \nabla L(x-y)[u_{\phi_k}^-(y)- u_{\phi_k}^-(x)] \ dx \ dy \right| \\
    &\leq C^k \iint\limits_{|x-y|\leq 8} \chi_{\{u<\phi_k\}}  |x-y|^{-(n-2\sigma+\g/(1+\g))} \ dx \ dy  \\
    &\quad  + \zeta \iint\limits_{|x-y|\leq 8} 
     \frac{[u_{\phi_k}^-(y)-u_{\phi_k}^-(x)]^{2+\g}}{|x-y|^{n+2-2\sigma}} \ dx \ dy  \\
    & C^k \quad + \iint\limits_{|x-y| > 8} \chi_{\{u<\phi_k\}} |x-y|^{-(n+1-2\sigma-\tau)}\ dx \ dy \\
   \end{aligned}
  \]
 The third term is bounded by 
  \[
   C \int_{\R^n} \chi_{\{u<\phi_k\}} \ dx
  \]
 provided $\tau <1-2\sigma$ as well as the first term provided again that $2\sigma > \g/(1+\g)$. The second term can be bounded as before
 by splitting the region of integration over $A_k$ and $A_k^c$ and absorbing the region over $A_k$ into the left hand
 side.  
 
 We now turn our attention to the last term involving $u_{\phi_k}^+$. We first remark that the integral becomes 
  \[
   \int_{\R^n}\int_{\R^n} \chi_{\{u(x)<\phi_k(x)\}} 
        (((u+d)/(\phi_k+d))^{\g+1}-1)\nabla \phi_k  \nabla L(x-y)u_{\phi_k}^+(y) \ dx \ dy. 
  \]
 We first consider the set $|x-y|>8$. Since $u_{\phi_k}^+ \leq \overline{\Psi}$, 
  \[
   \begin{aligned}
    &\left|\quad \iint\limits_{|x-y|> 8} \chi_{\{u(x)<\phi_k(x)\}} 
        (((u+d)/(\phi_k+d))^{\g+1}-1)\nabla \phi_k  \nabla L(x-y)u_{\phi_k}^+(y) \ dx \ dy \right| \\
    &\leq  C^k \left| \quad \iint\limits_{|x-y|> 8} \chi_{\{u(x)<\phi_k(x)\}} 
        |x-y|^{-(n+1-2\sigma+\tau)} \ dx \ dy \right| \\
    &\leq C^k \int_{\R^n} \chi_{\{u<\phi_k\}} \ dx.
   \end{aligned}
  \]
 When $|x-y|<8$, 
 we make the further decomposition 
  \[
   \frac{|\nabla \phi_k(x)|}{\phi_k(x)} |x-y| \leq \eta
  \]
 to absorb the integral by the nonnegative quantity \eqref{e:posi}. In the complement when 
  \[
   \frac{|\nabla \phi_k(x)|}{\phi_k(x)} |x-y| > \eta
  \]
 we use $\phi_k^{-1/m}C^k \geq |\nabla \phi_k|/\phi_k$ and integrate in $y$
  \[
   \left| \int_{B_8} \nabla L(x-y)u_{\phi_k}^+(y) dy \right|\leq \int_{\eta \phi_k^{1/m}C^{-k}}^8 \frac{r^{n-1}}{r^{n+1-2\sigma}} 
    \leq \max\{C, (\eta C^k)^{2\sigma -1} \phi_k^{(2\sigma-1)/m}\}. 
  \]
 The remainder of the terms are bounded by $|\nabla \phi_k|\leq C^k \phi_k^{1-1/m}$
 By multiplying by the term $\chi_{\{u<\phi\}}$ and integrating, we end up in the worst case with  
  \[
   C^k \int_{\R^n} \chi_{\{u<\phi_k\}} \phi_k^{1-1/m +(2\sigma -1)/m} \ dx \leq C^k \int_{\R^n} \chi_{\{u<\phi_k\}} \ dx,
  \]
 since $m \geq 2$.
 The last term to consider is the local spatial term. We use Cauchy-Schwarz
 \[
   \begin{aligned}
    &\delta\int_{B_R} \chi_{\{u<\phi_k\}}\nabla \left[\left(\frac{u+d}{\phi_k+d} \right)^{\g}-1 \right] (u+d)\nabla u\\
    &=   \delta \int_{B_R} \g \left( \frac{u+d}{\phi_k+d} \right)^{\g} \frac{|\nabla u|^2}{\phi_k+d}  \\
    &\quad - \delta \int_{B_R} \g \left( \frac{u+d}{\phi_k+d} \right)^{\g+1} \nabla u \nabla \phi_k  \\
    &\geq (1-\eta) \delta \g \int_{B_R}  \left( \frac{u+d}{\phi_k+d} \right)^{\g} |\nabla u|^2 \\
    &- C\delta\int_{B_R}\left( \frac{u+d}{\phi_k+d} \right)^{\g+2} |\nabla \phi_k|^2 \chi_{\{u < \phi_k\}} \\
    &\geq (1-\eta) \delta \g \int_{B_R}  \left( \frac{u+d}{\phi_k+d} \right)^{\g} |\nabla u|^2 
      - C^k\delta\int_{B_R}  \chi_{\{u < \phi_k\}} \ dx \\
   \end{aligned}
  \]

 
 Retaining the energy from \eqref{e:leftene} on the left hand side and moving everything else to the right hand side which is bounded by
 \eqref{e:rhs},
 our energy inequality becomes 
  \begin{equation}  \label{e:energy1}
   \begin{aligned}
    & c \int_{V_k} \left( \int_{-2-2^{-k}}^T (u-\phi)_-^{\frac{2}{1-\alpha}}\right)^{1-\alpha}  \\
    &\quad  + c \int_{-2-2^{-k}}^T \left( \int_{V_k} (u_{\phi_k}^-)^{n(2+\g)/(n-2+2\sigma)} \right)^{(n-2+2\sigma)/n} \\
    &\leq C^k\int_{\R^n} \int_{-\infty}^T \chi_{\{u<\phi_k\}}  
      \leq C^k\iint\limits_{U_{k-1}} \chi_{\{u<\phi_k\}} 
   \end{aligned}
  \end{equation}
  
 
 \textbf{Fourth Step: The nonlinear recursion relation.}
 We now (as in \cite{acv15}) use H\"older's inequality twice with the relations
  \[
   \frac{\beta}{p_1} + \frac{1-\beta}{p_2}= \frac{1}{p}= \frac{\beta}{p_3}+ \frac{1-\beta}{p_4}.
  \]
 for a function $v$ to obtain 
  \[
   \begin{aligned}
    \int \int v^p  & \leq \int \int v^{p \beta} v^{p (1-\beta)} \\
    & \leq  \int \left(\int v^{p_1} \right)^{p\beta/p_1} \left(\int v^{p_2} \right)^{p(1-\beta)/p_2} \\
    & \leq  \left( \int \left(\int v^{p_1} \right)^{p_3/p_1}\right)^{\beta p / p_3} 
            \left( \int \left(\int v^{p_2} \right)^{p_4/p_2}\right)^{(1-\beta) p / p_4}
   \end{aligned}
  \]
 We now choose
  \[
   p_1 =2, \quad p_2 = \frac{n(2+\g)}{n-(2-2\sigma)}, \quad p_3 = \frac{2}{1-\alpha}, \quad p_4 = 2+\g, 
  \]
 so that if $r=2-2\sigma$
  \[
   p = 2\frac{r+\alpha n (2+\g)/2}{(1-\alpha)r + \alpha n}\quad \text{and}\quad \beta = \frac{r}{r+\alpha n (2+\g)/2}.
  \]
 We now use H\"older's inequality one more time to obtain
  \[
   \begin{aligned}
     \left( \int \int v^p \right)^{b/p}
     &\leq \left( \int \left(\int v^{p_1} \right)^{p_3/p_1}\right)^{\beta b / p_3} 
            \left( \int \left(\int v^{p_2} \right)^{p_4/p_2}\right)^{(1-\beta) b / p_4} \\
    &\leq  \frac{1}{\omega}\left( \int \left(\int v^{p_1} \right)^{p_3/p_1}\right)^{\beta b \omega / p_3} 
            + \frac{\omega-1}{\omega}\left( \int \left(\int v^{p_2} \right)^{p_4/p_2}\right)^{\frac{(1-\beta) b \omega}{ p_4(\omega-1)}}     
   \end{aligned}         
  \]
 We choose 
  \[
   \omega = \frac{(2+\g \beta)}{(2+\g)\beta} \quad \text{and} \quad b = \frac{2(2+\g)}{2+\g \beta} < p
  \]
 so that 
 \begin{equation}  \label{e:minkowski}
   \begin{aligned}
    & \left( \int \int v^p \right)^{b/p} \\
    &\leq  \frac{1}{\omega}\left( \int \left(\int v^{2} \right)^{1/(1-\alpha)}\right)^{1-\alpha} 
            + \frac{\omega-1}{\omega}\left( \int \left(\int v^{\frac{n(2+\g)}{n-r}} \right)^{\frac{n-r}{n}}\right)  \\     
    &\leq \frac{1}{\omega}\int \left( \int  v^{\frac{2}{1-\alpha}} \right)^{1-\alpha} 
            + \frac{\omega-1}{\omega}\left( \int \left(\int v^{\frac{n(2+\g)}{n-r}} \right)^{\frac{n-r}{n}}\right),         
   \end{aligned}        
  \end{equation}
 where we used Minkowski's inequality in the last inequality. 
 Substituting $u_{\phi_k}^-$ for $v$ in \eqref{e:minkowski} and utilizing \eqref{e:leftene} we obtain
  \begin{equation}  \label{e:iteration}
   \left( \iint\limits_{U_k} (u-\phi_k)_-^p \right)^{b/p} \leq C^k\iint\limits_{U_{k-1}} \chi_{\{u<\phi_k\}}. 
  \end{equation}

 We first recall that $\phi_{k-1}\geq \phi_k + (1-\mu_0)2^{-k}$. We now utilize Tchebychev's inequality 
  \[
   \iint\limits_{U_k} \chi_{\{u<\phi_k\}} \leq (2^k/(1-\mu_0))^p \iint\limits_{U_{k-1}} (u-\phi_{k-1})_-^p. 
  \]
 Combining the above inequality with \eqref{e:iteration} we conclude
  \[
   \iint\limits_{U_k} (u-\phi_k)_-^p \leq C^k\left( \iint\limits_{U_{k-1}} (u-\phi_{k-1})_-^p  \right)^{p/b}.
  \]
 If we define
  \[
   M_k := \iint\limits_{U_k} (u-\phi_k)_-^p,
  \]
 Then 
  \[
   M_k \leq C^k M_{k-1}^{p/b}. 
  \]
 Since $p>b$, if $M_0$ is sufficiently small - depending on $C$ and $p/b$ - we obtain that $U_k \to 0$, and hence $u\geq \mu_0$.

\end{proof}

  We now prove Lemma \ref{l:pullup} in the case when $D(x)=(d_1 x + d_2)$ with $d_1 \leq 2$ and $d_2 \geq 1/2$. This is actually much simpler because we 
  can utilize the test function $-(u-\phi_k)_-$ as when dealing with a linear equation. 
  
  \begin{proof}
   We choose as the test function $-(u-\phi_k)_-$. Notice that 
    \[
     \nabla u_{\phi_k}^- (d_1 u + d_2) = \nabla u_{\phi_k}^- d_1 u + \nabla u_{\phi_k}^- d_2.
    \]
   The fact that $d_2\geq 1/2$ gives a nondegenerate linear term which we utilize. From the computations in \cite{acv15} we then have
    \[
     \begin{aligned}   
     &\int_{\R^n} \int_{-\infty}^0 u_{\phi_k}^- \D_t^{\alpha} u  \\
      &\geq c \int_{\R^n} \left( \int_{-\infty }^0 \left(u_{\phi_k}^-\right)^{\frac{2}{1-\alpha}}\right)^{1-\alpha} \\
      &\quad  - C^k \int_{\R^n}\int_{-\infty}^0 \chi_{\{u<\phi_k\}} \ dt,
     \end{aligned}
    \]
   and even more importantly 
    \begin{equation}   \label{e:better}
     \begin{aligned}
      &\int_{-\infty}^0 \int_{\R^n} \nabla u_{\phi_k}^- \nabla (-\Delta)^{-\sigma} u \\
      &\geq c\int_{-\infty}^0 \int_{\R^n} \int_{\R^n} \frac{[u_{\phi_k}^-(x)-u_{\phi_k}^-(y)]^2}{|x-y|^{n+2-2\sigma}} \ dx \ dy \ dt \\
      &- C^k \int_{-\infty}^0 \int_{\R^n} \chi_{\{u<\phi_k\}} \ dx \ dt. 
     \end{aligned}
    \end{equation}
   Notice that in the above inequality we have the power $|\cdot|^2$ rather than $|\cdot|^{2+\g}$. 
   
   We now show how to bound the terms invovling $d_1 u$. 
    \[
     -\nabla u_{\phi_k}^-[u_{\phi_k}^+ - u_{\phi_k}^- + \phi_k]
      = \nabla[(u_{\phi_k}^-)^2/2 - u_{\phi_k}^- \phi_k] + u_{\phi_k}^- \nabla \phi_k
      = (1)+(2)
    \]
   Then multiplying $(1)$ by $\nabla (-\Delta)^{-\sigma} u$ and integrating over $\R^n$ we have
    \[
     \B((u_{\phi_k}^-)^2 /2 -u_{\phi_k}^- \phi_k, u_{\phi_k}^+ - u_{\phi_k}^- + \phi_k ). 
    \]
   Now 
    \[
     \begin{aligned}
      &[(u_{\phi_k}^-)^2 /2 -u_{\phi_k}^- \phi_k](x) - [(u_{\phi_k}^-)^2 /2 -u_{\phi_k}^- \phi_k](y)  \\
      &= \frac{1}{2}[u_{\phi_k}^- (y) -u_{\phi_k^-}(x)][\phi_k(x)+\phi_k(y)-u_{\phi_k}^- (x)- u_{\phi_k}^- (y)]\\
      &\quad - \frac{1}{2}[u_{\phi_k}^- (y) +u_{\phi_k^-}(x)][\phi_k(x)-\phi_k(y)] \\
      &=(1a)+(1b)
     \end{aligned}
    \]
   We write $u=u_{\phi_k}^+ -u_{\phi_k}$. We break up our set into the two regions
    \[
      F_k:= \{(x,y): |u_\phi^-(x)-u_\phi^-(y)| \geq 2|\phi(x)-\phi(y)| \} \\
    \]
  We notice that on the set $F$ we have that 
   \[
    \phi(x)+\phi(y)-u_\phi^-(x)-u_\phi^-(y) \geq |u_\phi^-(x)-u_\phi^-(y)|/2.
   \]
  Then integrating over $F$ we have for the term $(1a)$ with right term $-u_{\phi_k}^-$ 
   \[
    \begin{aligned}
    &-\iint\limits_{F_k} 
    \frac{([(u_{\phi_k}^-)^2 /2 -u_{\phi_k}^- \phi_k](x) - [(u_{\phi_k}^-)^2 /2 -u_{\phi_k}^- \phi_k](y))(u_{\phi_k^-(x)}-u_{\phi_k}^-(y))}{|x-y|^{n+2-2\sigma}} \\
    &\geq \iint\limits_{F_k} \frac{|u_{\phi_k^-(x)}-u_{\phi_k}^-(y)|^3}{|x-y|^{n+2-2\sigma}} \geq 0.
    \end{aligned}
   \]
  This is the nonnegative energy piece which we actually do not need having obtained a better piece in \eqref{e:better}. All of the remaining terms in $(1)$
  can be 
  bounded by breaking up the region of integration over $F_k,F_k^c$. Over $F_k$ we use H\"older's inequality with $p=2$ rather than with $p=2+\g$ and absorb
  the small pieces by the term in \eqref{e:better}. We use the same methods as before to bound the integration over $F_k^c$. Bounding the term $(2)$ is done as 
  before with slightly easier computations.   
  
  The local spatial term is bounded in the usual manner. 
  \end{proof}



\section{Pull-down}   \label{s:pulldown}
 
 In this section we prove Lemma \ref{l:pulldown}.
 We will need the following estimate that is analogous to Lemma \ref{l:control}.
 \begin{lemma}    \label{l:pulldest}
   Let $u,\phi$ be two functions such that $1/2\leq \phi \leq u \leq 1$. Let $0<\gamma<1$ be a constant. 
   If $|u(x)-u(y)|\geq 4|\phi(x)-\phi(y)|$, then 
    \begin{equation} \label{e:pcontrol1}
     c_1 |u_{\phi}^+(x)-u_{\phi}^+(y)|^{1+\g} 
      \leq \left| \frac{u^{\g+1}(y)}{\phi^{\g}(y)} - \frac{u^{\g+1}(x)}{\phi^{\g}(x)} \right|
      \leq c_2 |u_{\phi}^+(x)-u_{\phi}^+(y)|.
    \end{equation}
   Also, if  
    \[
     0 \leq  \frac{u^{\g+1}(y)}{\phi^{\g}(y)} - \frac{u^{\g+1}(x)}{\phi^{\g}(x)} 
    \] 
   then $0\leq (u-\phi)_-(x)-(u-\phi)_-(y)$. 
   
   If instead we assume $|u(x)-u(y)|\leq 4|\phi(x)-\phi(y)|$, then 
    \begin{equation} 
      \left| \frac{u^{\g+1}(y)}{\phi^{\g}(y)} - \frac{u^{\g+1}(x)}{\phi^{\g}(x)} \right|
      \leq 14 |\phi(x)-\phi(y)|.
    \end{equation}
  \end{lemma}

  The proof is similar to the proof of Lemma \ref{l:control}. In this case since $u>\phi$ one uses the bound above on $u$ and the fact that $\phi$ is 
  bounded by below. 

 \begin{proof}[Proof of Lemma \ref{l:pulldown}]
  The proof is nearly identical. We mention the differences. We only consider $D(u)=u$ since the modifications for handling $D(u)=d_1 u +d_2$ have
   already been shown in the proof of Lemma \ref{l:pullup}. 
   We consider a similar test function 
   \[
    F(x)= \frac{1}{\g+1}(1+x)^{\g+1} -x + \frac{1}{\g+1}
   \]
  We then utilize
  \begin{equation} \label{e:test2}
    F'(u_{\phi_k}^+/\phi_k) = \left(1+ \frac{(u-\phi_k)_+}{\phi} \right)^{\g} -1
   \end{equation}
 This time we consider the same test functions $\theta_k(x)$ in the space variable, but this time we multiply only by a single cut-off in time $\xi_0(t)$. 
 We define our $\phi_k$ as
  \[
   \phi_k:= \overline{\Psi}(x,t) -\xi_0(t)\theta_k(x/2).
  \] 
 To obtain the same estimate in time we only need to recognize that $\phi_k$ is now decreasing in time and bounded by below by $1/2$. 
 
  \begin{equation}  \label{e:pultime} 
   \begin{aligned}
    &F' \left(\frac{u_{\phi}^+(t)}{\phi(t)} \right) \left[u_{\phi}^+(t)-u_{\phi}^+(s)\right] \\
    &= \phi(t)F' \left(\frac{u_{\phi}^+(t)}{\phi(t)} \right) \frac{u_{\phi}^+(t)-u_{\phi}^+(s)}{\phi(t)} \\ 
    &= \phi(t)F' \left(\frac{u_{\phi}^+(t)}{\phi(t)} \right) \left[\frac{u_{\phi}^+(t)}{\phi(t)}-\frac{u_{\phi}^+(s)}{\phi(s)}\right] \\
    &\quad + \phi(t)u_{\phi}^s F' \left(\frac{u_{\phi}^+(t)}{\phi(t)} \right) \left[\frac{1}{\phi(t)}-\frac{1}{\phi(s)}\right] \\
    &\geq \frac{1}{2}\left[F(u_{\phi}^+(t)/\phi(t)) -  F(u_{\phi}^+(s)/\phi(s)) \right] \\
    &\quad + \frac{\g}{2}[u_{\phi}^+(t)-u_{\phi}^+(s)]^2 \\
    &\quad - C u_{\phi}^+(s) (t-s) 
   \end{aligned}
  \end{equation}  
 The negative constant comes from the fact that $\phi^{-1}$ is Lipschitz. Then everything proceeds as
 before. Since our cut-off is bounded by below our $L^p$ norm in time occurs over all of $(-\infty,0)$. We obtain as before
 
  \[
   \begin{aligned}
    &c\int_{\R^n} \left( \int_{-\infty}^0 (u-\phi)_+^{\frac{2}{1-\alpha}}\right)^{1-\alpha}
     +\int_{-\infty}^0 \int_{\R^n} u \nabla F'\left(\frac{u_{\phi}^+}{\phi} \right) \nabla (-\Delta)^{-\sigma} u \\
    &\leq   \int_{\R^n} \int_{-\infty}^T \chi_{\{u>\phi\}} \ dt \ dx
   \end{aligned}
  \]
 
 The spatial portion of the problem is handled exactly as before.

 \end{proof}

\section{Decrease in Oscillation}  \label{s:oscillation}
  We define 
   \[
    F(t,x):= \frac{1}{4}\sup\{-1, \inf\{0,|x|^2-9\}\} +\frac{1}{4}\sup\{-1, \inf\{0,|t|^2-9\}\}. 
   \]
  We point out that $F$ is Lipschitz, compactly supported in $[-3,0]\times B_3$ and equal to $-1/2$ in $[-2,0]\times B_2$. 
  We also define for $0<\lambda<1/4$,
   \[
    \psi_{\lambda}(t,x):= ((|x|-\lambda^{-1/\nu})^{\nu}-1)_+ + ((|t|-\lambda^{-1/\nu})^{\nu}-1)_+ \text{ for } |t|,|x|\geq \lambda^{-1,\nu}
   \]
  and zero otherwise. The value of $\nu$ will be determined later. 
  Finally, we define for $i \in \{0,1,2,3,4\}$ 
   \[
    \phi_i = 1 + \psi_{\lambda^3} + \lambda^i F.
   \]
  Then $1/2\leq \phi_0 \leq \ldots \leq \phi_4 \leq 1$ in $\Gamma_4$.

  \begin{lemma}  \label{l:decintime}
   Let $\kappa$ be the constant defined in Lemma \ref{l:pulldown}. 
   Let $u$ be a solution to \eqref{e:2delaprox}. There exists a small constant $\rho>0$ depending only on $n,\sigma, \alpha$ and $\lambda_0$ 
   depending only on $n,\sigma, \alpha, \rho, \delta$
   such that for any solution $u$ defined in $(a,0)\times \R^n$ with $a<-4$ and 
    \[
     - \psi_{\lambda^3} \leq u(t,x) \leq 1 + \psi_{\lambda^3} \text{  in  } (a,0)\times \R^n
    \] 
   with $\lambda \leq \lambda_0$, and $f\leq \lambda^3$, then if 
    \[
     |\{u<\phi_0\}\cap (B_1 \times (-4,-2))| \geq \rho,
    \] 
   then
    \[
     |\{u>\phi_4\} \cap (\R^n \times (-2,0))| \leq \kappa.
    \]
  \end{lemma}

  \begin{proof}
   We will show the computations for $D(u)=u$. The general situation is handled as before as in Lemma \ref{l:pullup}. 
   
   \textbf{First Step: Revisiting the energy inequality.}
   We return again to the energy inequality. This time, however, we will make use of the nonnegative terms. We seek to obtain a bounde on the 
   right hand side of the form $C\lambda^{(2+\g)/(1+\g)}$. 
   
   We now consider the test function as 
   in \eqref{e:test2}, but with cut-off $\phi_1$. If $u>\phi_i$ , then $1/2\leq \phi_i \leq u\leq 1$, and so 
    \[
     F'(u/\phi_i) = \chi_{\{u>\phi_i\}}\frac{u^{\g}-\phi^{\g}}{\phi^{\g}}\leq 2\gamma u_{\phi_i}^+ \leq 2\gamma \lambda^{2i}.  
    \]
   To take care of the piece in time
    we first note that $\phi_1$ is Lipschitz in time for $t \in [0,4]$ with Lipschitz constant $2\lambda$. Then as before
    \[
      \int F'(u/\phi_1) \D_t^{\alpha} u   \ dt = \int F'(u/\phi_1) \D_t^{\alpha} (u_{\phi_1}^+ - u_{\phi_1}^- + \phi) = (T1)+(T2)+(T3).
    \]
   For $(T1)$ we return to the inequality \eqref{e:pultime} and utilize the Lipschitz nature of $\phi_1^{-1}$, to obtain   
    \[
     \begin{aligned}
      \iint F'(u/\phi_1) \D_t^{\alpha} u_{\phi_1}^+ &\geq c \left( \int_{-\infty}^0 (u-\phi_1)_+^{\frac{2}{1-\alpha}}\right)^{1-\alpha} \\
      &\quad +\int_{-\infty}^0 \int_{-\infty}^t \phi_1(t) u_{\phi_1}^{s} F'(u_{\phi_1}^+(t)/\phi(t))(\phi_1^{-1}(t)-\phi_1^{-1}(s)) \ dt  \\
      &\geq c  \left( \int_{-\infty}^0 (u-\phi_1)_+^{\frac{2}{1-\alpha}}\right)^{1-\alpha}\\
      &\quad - C \int_{-4}^0 \lambda^2.  
     \end{aligned}
    \]
   The nonnegative piece (T2) will be utilized in the second step of this proof. For $(T3)$ we note that
   since $\phi_i$ is decreasing, we have 
   \[
    0 \leq - \D_t^{\alpha} \phi_i \leq -\Lambda^{-1} D_t^{\alpha} \phi_i = -D_t^{\alpha} \psi_{\lambda^3} -D_t^{\alpha} \lambda^i F.
   \]
   Clearly, $-D_t^{\alpha} \lambda^i F \leq C\lambda^i$ for $t\leq 0$ from the Lipschitz nature of $F$. For $-4\leq t\leq 0$ we have
    \[
     -D_t^{\alpha} \psi_{\lambda^3} \leq \int_{-\infty}^{\lambda^{-1/\nu}} \frac{|s|}{|s|^{1+\alpha}} \leq C_{\alpha} \lambda^{(\alpha - \nu)/(\nu)}.
    \]
   We therefore pick $\nu$ small enough that $(\alpha-\nu)/\nu>2$. 
    \[
      \int_{-\infty}^0 F'(u/\phi_1) \D_t^{\alpha} \phi_1 \leq C \int_{-4}^0 \lambda u_{\phi}^+ \leq C\lambda^2.
    \]
   Our energy inequality becomes 
    \[
     \begin{aligned}
      &c \left( \int_{-\infty}^0 (u-\phi)_+^{\frac{2}{1-\alpha}}\right)^{1-\alpha} \\
      &\quad + c\int_{\R^n} \iint \frac{[u_{\phi_1}^+(t)-u_{\phi_1}^+(s)][u_{\phi_1}^-(t)- u_{\phi_1}^-(s)]}{(t-s)^{1+\alpha}}  \\
      &\quad + \text{  ``Spatial Terms''  } \leq C \lambda^2 + \iint f u_{\phi_1}^+. 
     \end{aligned}
    \]
   Since $f\leq \lambda^3$, everything is bounded on the right hand side by $C\lambda^2$. 
   
   We now turn our attention to the elliptic portion.  We consider the terms $(1a),(1bi),(1bii),(1c),(2)$ as the analogous terms for 
   those defined in the proof of Lemma \ref{l:pullup}. 
   As before we obtain a nonnegative energy from the term $(1bi)$. Everything else we will absorb into this energy or bound by $C\lambda^{(2+\g)/(1+\g)}$. 
   The term from $(1bi)$ over $A_1^c$ is bounded as follows: 
    \[
     \begin{aligned}
     &c_{n,\g,\sigma}\iint_{A_1^c} \frac{[\phi_1(x)-\phi_1(y)][u_{\phi_1}^+(x)-u_{\phi_1}^+(y)]}{|x-y|^{n+2-2\sigma}} \ dx \ dy \\
     &\leq C \iint_{A_1^c} \frac{[\phi_1(x)-\phi_1(y)]^2 X_{x,y}}{|x-y|^{n+2-2\sigma}} \ dx \ dy 
     \end{aligned}
    \]
   We have the following inequality from the computations given in \cite{ccv11}:
    \begin{equation}  \label{e:wingbound} 
     \iint \frac{[\phi_1(x)-\phi_1(y)]^{(2+\g)/(1+\g)} X_{x,y}}{|x-y|^{n+2-2\sigma}} \ dx \ dy \leq C \lambda^{(2+\g)/(1+\g)}, 
    \end{equation}
   with $(2-2\sigma-2\nu)/\nu\geq 2$. 
   In particular, \eqref{e:wingbound} will hold for $\g=0$. 
   For the term $(1bii)$ we break up the region of integration into $A_1$ and $A_1^c$. On $A_1$ we use H\"older's inequality as before 
    \[
     \begin{aligned}
     &c_{n,\g,\sigma}\iint_{A_1} \frac{[\phi_1(x)-\phi_1(y)][u_{\phi_1}(x)-u_{\phi_1}(y)]}{|x-y|^{n+2-2\sigma}} \ dx \ dy \\
      &\leq C  \iint_{A_1} \frac{[\phi_1(x)-\phi_1(y)]^{(2+\g)/(1+\g)}}{|x-y|^{n+2-2\sigma}}X_{x,y} \ dx \ dy \\
      &\quad + \eta \iint_{A_1} \frac{[u_{\phi_1}(x)-u_{\phi_1}(y)]^2}{|x-y|^{n+2-2\sigma}} \ dx \ dy \\
      &\leq C\lambda^{(2+\g)/(1+\g)} + \eta \iint_{A_1} \frac{[u_{\phi_1}(x)-u_{\phi_1}(y)]^2}{|x-y|^{n+2-2\sigma}} \ dx \ dy \\
     \end{aligned}
    \]
   The last term is absorbed into the left hand side. The other term is bounded again from \eqref{e:wingbound}. 
   The term $(1c)$ is bounded in exactly the same way. $(1a)$ is nonnegative and will be utilized later. 
   We now turn our attention to the term $(2)$. We rewrite $u=u_{\phi_1}^+ - u_{\phi_1}^- + \phi_1$.  The term involving $u_{\phi_1}^+$ 
   with $|x-y|\leq \eta$ is absorbed by the nonnegative term $(1a)$ on the left hand side. We now utilize the inequalities: 
    \begin{itemize}
      \item $u_{\phi_1}^+ \leq \lambda $
      \item $|\nabla L(x-y)| \approx 1/|x-y|^{n+1-2\sigma}$
      \item $|\nabla \phi_1| \leq C $  for all  $x$ 
      \item $|\nabla \phi_1| \leq C\lambda $ in the support of $u_{\phi_1}^+$
      \item $\chi_{\{u>\phi_1\}} [(u/\phi_1)^{1+\g}-1] \leq 4 u_{\phi_1}^+ \leq 4 \lambda. $
    \end{itemize}
   Then 
    \[
     \begin{aligned}
      &\left| \int_{\R^n}\int_{\R^n}\frac{\g}{\g+1} \chi_{\{u<\phi\}} ((u/\phi)^{\g+1}-1) \nabla \phi_1  \nabla L(x-y)[u(y)-u(x)] \ dx \ dy \right| \\
      &\leq C \lambda^2
       \left| \int_{\R^n}\int_{\R^n} \chi_{\{u<\phi\}} \nabla L(x-y)[u(y)-u(x)] \ dx \ dy \right| \\ 
     \end{aligned}
    \] 
   These terms are all bounded as before. Notice that we have $\lambda^2$ on the outside. Then all nonloacl terms on the right hand side are bounded by
   $C\lambda^{(2+\g)/(1+\g)}$. The local term div$(D(u)\nabla u)$ is handled in the usual manner by use of Cauchy-Schwarz.

   \textbf{Second Step: Using the ``good'' spatial piece.}
   We now utilize the two nonnegative pieces. From Proposition \ref{p:gamma} we have 
    \[
      \left[ \left(\frac{u}{\phi_1}\right)^{\g+1}-\phi_1  \right]_+ \geq 4 (u_{\phi_1}^+)^{1+\g}.
    \]
    Then we conclude that 
    \begin{equation}  \label{e:phi1}
     \begin{aligned}
      &\int_{\R^n} \int_{-4}^0 \int_{-4}^0\frac{u_{\phi_1}^+(t)u_{\phi_1}^-(s)}{t-s}^{1+\alpha}   \\
      &\quad + \int_{-4}^0 \int_{\R^n}\int_{\R^n}\frac{(u_{\phi_1}^+)^{\g+1}(x)u_{\phi_1}^-(y)}{|x-y|^{n+2-2\sigma}}\\
      &\leq C \lambda^{(2+\g)/(1+\g)}. 
     \end{aligned}
    \end{equation}
   Since we used $\Psi_{\lambda^3}$, replacing $\phi_1$ with $\phi_3$ we have the same inequality but with the bound $C\lambda^{3(2+\g)/(1+\g)}$.  

  We now show how the inequality \eqref{e:phi1} and its analogue for $\phi_3$ are enough to prove the remainder of the Lemma as in \cite{acv15}. 
  We note that 
  for the proof as written in \cite{acv15} to work we need $3(2+\g)/(1+\g)>5$ which is achieved for $\gamma$ small enough. We first utilize 
   \[
    \int_{-4}^0 \int_{\R^n}\int_{\R^n}\frac{(u_{\phi_1}^+)^{\g+1}(x)u_{\phi_1}^-(y)}{|x-y|^{n+2-2\sigma}}\\
    \leq C \lambda^{(2+\g)/(1+\g)}.
   \]
  From our hypothesis
   \[
    |\{w<\phi_0\} \cap ((-4,-2)\times B_1)| \geq \rho.
   \]  
  Then the set of times $\Sigma \in (-4,-2)$ for which $|\{u(t,\cdot)<\phi_0\} \cap B_1| \geq \rho/4$ has atleast measure $\rho/(2|B_1|)$. 
  And so
   \[
    C\lambda^{(2+\g)/(1+\g)} 
     \geq c\rho \int_{\Sigma} \int_{\R^n} (u_{\phi_1}^+)^{1+\g} \ dx \ dt 
   \]
  Now $(\{u-\phi_2>0\}\cap (\Sigma \times B_2)) \subset (\{u-\phi_1>\lambda/2\}\cap (\Sigma \times B_2))$, and so from Tchebychev's inequality 
   \[
    |\{u-\phi_2>0\}\cap (\Sigma \times B_2)| \leq \frac{C}{\rho} \lambda^{\frac{2+\g}{1+\g}-(1+\g)}.
   \]
  The exponent on $\lambda$ is positive for $\g$ small enough. We write this as
   \[
   |\{u\leq \phi_2\}\cap (\Sigma \times B_2)| \geq |\Sigma \times B_2| - \frac{C}{\rho} \lambda^{\frac{2+\g}{1+\g}-(1+\g)}
    \geq \rho/2 - \frac{C}{\rho} \lambda^{\frac{2+\g}{1+\g}-(1+\g)}.
   \]
  This will be positive for $\lambda$ small enough depending on $n,\sigma, \alpha, \gamma, \rho$. 
  
  The proof then proceeds just as in \cite{acv15} where we then utilize $3(2+\g)/(1+\g)>5$ as well as the analogue of \eqref{e:phi1}
  for $\phi_3$.  
  \end{proof}

 This next Lemma will imply Lemma \ref{l:2down}. 
 For this next lemma we define
  \[
   \psi_{\tau, \lambda} = ((|x|-1/\lambda^{4/\sigma})^{\tau}-1)_+ + ((|t|-1/\lambda^{4/\alpha})^{\tau}-1)_+ 
  \]

 \begin{lemma}  \label{l:oscdec}
  Given $\rho>0$ there exist $\tau>0$ and $\mu_1$ such that for any solution to \eqref{e:2delaprox} in $\R^n \times (a,0)$ with $a<-4$ and 
  $|f|\leq \lambda^3$ satisfying
   \[
    - \psi_{\tau, \lambda} \leq u \leq 1 + \psi_{\tau, \lambda},
   \]
  If 
   \[
    |\{u<\phi_0\} \cap (B_1 \times (-4,-2))| > \rho,
   \]
  then 
   \[
    \sup_{B_1 \times (-1,0)} u \leq 1- \mu_1.
   \]
 \end{lemma}
 
 \begin{proof}
  We consider the rescaled function $w(t,x)= (u-(1-\lambda^4))/\lambda^4$. We fix $\tau$ small enough such that
   \[
    \frac{(|x|^{\tau}-1)_+}{\lambda^4} \leq (|x|^{\sigma/4}-1)_+ \text{  and  }
    \frac{(|t|^{\tau}-1)_+}{\lambda^4} \leq (|t|^{\sigma/4}-1)_+
   \]
  Then $w$ satisfies equation \eqref{e:2delaprox} with $D_2(x) = D_1(\lambda^4 x + (1-\lambda^4))$ where $D_1$ is the coefficient for the equation $u$ satisfies. 
  From our hypothesis and Lemma \ref{l:decintime}
   \[
    |\{w>1/2\} \cap (B_2 \times [-2,0])| = |\{u>\phi_4\} \cap (B_2 \times [-2,0])| \leq \rho.
   \]
  Then from Lemma \ref{l:pulldown} we conclude that $w\leq 1-\mu_1$ on $(-1,0)\times B_1$, and so $u\leq 1-\lambda^4 \mu_1= 1-\mu_2$.  
 \end{proof}

\section{Proof of Regularity}   \label{s:regularity}
 With Lemmas \ref{l:pullup}, \ref{l:pulldown}, and \ref{l:2down} we are ready to finish the proof of Theorem \ref{t:continuity}. We first mention
 that solutions of \eqref{e:main} satisfy the following scaling property: If $u$ is a solution on $(a,0)\times \R^n$, then $v(t,x)=Au(Bt,Cx)$ is a solution
 on $(a/B,0)\times \R^n$ if $A=B^{\alpha}C^{2\sigma-2}$.  
 The method of proof is given in \cite{cfv13} which we now briefly outline. We take any point 
 $p=(x,t)\in \R^n \times (a, T) $ and prove that $u$ is H\"older continuous around $p$. The H\"older continuity exponent will depend only
 on $\alpha, \sigma, n$. The constant will depend on the $L^{\infty}$ norm of $u,f$ and  on the $C^2$ norm of $u(a,x)$.
   By translation we assume that $p=(0,0)$. By scaling we assume that 
 $0\leq u(t,x)\leq \overline{\Psi}(t,x)$ and $|f|\leq \lambda^3$ for $\lambda$ as defined in  
 
 We now take a positive constant $M<1/4$ such that for $0<K\leq M$
  \[
   \frac{1}{1-\mu_2/2} \psi_{\tau,\lambda^3}(Kt,Mx).
  \]
 $M$ will depend only on $\lambda, \mu_2$ and $\tau>0$. During the iteration we have the following alternative. 
 
 \textit{Alternative 1.}
 Suppose that we can apply Lemma \ref{l:2down} repeatedly. We then consider the rescaled functions 
  \[
   u_{j+1}(t,x)= \frac{1}{1-\mu_2/4} u_j(M_1t, Mx), \quad M_1 = \left(\frac{M^{2-2\sigma}}{1-\mu_2/4} \right)^{1/\alpha}. 
  \]
 Notice that $M_1 < M$. All the $u_j$ satisfy the same equation. If we can apply Lemma \ref{l:2down} at every step, then $u_j\leq 1-\mu_2$ on 
 the cylinder $\Gamma_1$. This implies H\"older regularity around $p$ and also implies $u(p)=0$ .
 
 \textit{Alternative 2.}
 If at some point the assumption \eqref{e:2down} fails, then we are in the situation of Lemma \ref{l:pullup} and 
  \[
   0<\mu_0 \leq u_j(t,x)\leq 1.
  \]
 Scaling the above situation our equations will have $D(u)=d_1 u + d_2$ with $d_2>0$. We may then repeat the procedure since Lemma \ref{l:pullup}
 and \ref{l:2down} apply also in this situation.

\section{appendix}

  \begin{proof}[Proof of Lemma \ref{l:control}]
   Since throughout the paper we only require $\gamma$ small when $\sigma$ is small, we will prove the Lemma for $\gamma=1/k$ for $k\in \N$. 
   Now we assume without loss of generality that 
    \[
     \frac{u^{\g+1}(y)}{\phi^{\g}(y)} - \frac{u^{\g+1}(x)}{\phi^{\g}(x)} \geq 0. 
    \]
   We first assume that $|u(x)-u(y)|\leq 4|\phi(x)-\phi(y)|$.
   We need to bound 
    \begin{equation}  \label{e:quotient}
     \frac{u^{\g+1}(y)}{\phi^{\g}(y)} - \frac{u^{\g+1}(x)}{\phi^{\g}(x)}. 
    \end{equation}
   We first notice that term above in \eqref{e:quotient} will be larger if we assume
   that $u(y)\geq u(x)$ and $\phi(x)\geq \phi(y)$ without changing $|u(x)-u(y)|$ and $|\phi(x)-\phi(y)|$. 
   Furthermore, the term in \eqref{e:quotient} will still be greater if $u(y)=\phi(y)$ and not 
   changing $|u(y)-u(x)|$. We are then looking for the bound
    \[
     u(y)- \frac{u^{1+\g}(x)}{\phi^{\g}(x)} \leq c_2|\phi(x)-\phi(y)| =c_2(\phi(x)-u(y)). 
    \]
   Thus, for a constant $l$ we need the bound 
    \begin{equation}  \label{e:mubound}
     l - \frac{(l-\upsilon)^{1+\g}}{(l+ \mu)^{\g}} \leq C \mu.
    \end{equation}
   Recalling that we are assuming $4|\phi(x)-\phi(y)|\geq |u(x)-u(y)|$ (or $\upsilon \leq 4\mu$)
   the above term is maximized when $\upsilon$ is largest or when $\upsilon=4\mu$.
   Now
    \[
     \begin{aligned}
      &l - \frac{(l-4\mu)^{1+\g}}{(l+ \mu)^{\g}} \\
      &=  \frac{l[(l+\mu)^\g -(l-4\mu)^{\g}]}{(l+\mu)^{\g}} + 4\frac{\mu(l-4\mu)^{\g}}{(l+\mu)^{-\g}} \\
      &:= L_1 + L_2
     \end{aligned}
    \]
   It is clear that
    \[
     L_2 \leq 4\mu.
    \]
   To control $L_1$ we first consider when $l-4\mu \leq l/2$. Then $l\leq 8\mu$ and it is clearly true that
    \[
     L_1 \leq 8 \mu.
    \]
   Now when $l-4\mu \geq l/2$, from the concavity of $x^\gamma$ we have
    \[
     L_1 \leq \frac{l}{l-4\mu}\frac{(l-4\mu)^{\g}}{(l+\mu)^{\g}}5\mu \leq 10 \mu.
    \]
   Then 
    \begin{equation}   \label{e:mu1}
     l - \frac{(l-\upsilon)^{1+\g}}{(l+ \mu)^{\g}} \leq 14 \mu,
    \end{equation}
   and \eqref{e:control2} is proven with constant $c_2=14$. 
  
   We now assume $4|\phi(x)-\phi(y)|\leq |u(x)-u(y)|$, and the left hand side of \eqref{e:mubound} 
   is maximized again when $4\mu=\upsilon$ and so we have
    \[
     l - \frac{(l-\upsilon)^{1+\g}}{(l+ \upsilon/4)^{\g}} \leq \frac{14}{4}\upsilon,
    \]
   which is just \eqref{e:mu1} rewritten with the substitution $\mu=\upsilon/4$. Then  
    \begin{equation}  \label{e:Abound1}
      \frac{u^{\g+1}(y)}{\phi^{\g}(y)} - \frac{u^{\g+1}(x)}{\phi^{\g}(x)}  \leq \frac{14}{4}|u(x)-u(y)|
     \leq \frac{14}{3} |u_{\phi}^-(x)-u_{\phi}^-(y)|.
    \end{equation}
   and the right hand side of \eqref{e:control1} is shown. 
   
   Now  
    \[
     \frac{u^{\g+1}(y)}{\phi^{\g}(y)} - \frac{u^{\g+1}(x)}{\phi^{\g}(x)}
      = \frac{u^{\g+1}(y)-u^{\g+1}(x)}{\phi^{\g}(y)} 
      - u^{\g+1}(x)\frac{\phi^{\g}(x)-\phi^{\g}(y)}{\phi^{\g}(x)\phi^{\g}(y)} := M_1 + M_2
    \]
   We suppose $\gamma = 1/k$. By factoring we have 
    \[
     \begin{aligned}
      |M_1| &= \frac{|u(y)-u(x)|}{\phi^{\g}(y)} \frac{\sum_{j=0}^k u^{(k-j)/k}(y)u^{j/k}(x)}{\sum_{j=0}^{k-1} u^{(k-1-j)/k}(y)u^{j/k}(x)} \\
            &\geq \frac{u(x)}{\phi^{\g}(y)} \frac{4|\phi(x)-\phi(y)|}{\sum_{j=0}^{k-1} u^{(k-1-j)/k}(y)u^{j/k}(x)} \\
      |M_2| &= \frac{u^{1+1/k}(x)}{\phi^{1/k}(x)\phi^{1/k}(y)} 
                  \frac{|\phi(x)-\phi(y)|}{\sum_{j=0}^{k-1} \phi^{(k-1-j)/k}(y)\phi^{j/k}(x)} \\
            &\leq \frac{u(x)}{\phi^{\g}(y)} \frac{|\phi(x)-\phi(y)|}{\sum_{j=0}^{k-1} u^{(k-1-j)/k}(y)u^{j/k}(x)} \\
     \end{aligned}
    \]
   Thus $M_2 \leq M_1 / 4$. Thus, $M_1$ is the dominant term. We then have from the convexity of $x^{\g+1}$
    \[
     \begin{aligned}
      M_1 + M_2 \geq M_1/2 &\geq \frac{1}{2} \frac{u^{\g+1}(y)-u^{\g+1}(x)}{\phi^{\g}(y)} \\
                           &\geq  \frac{1}{2} \frac{(u(y)-u(x))^{1+\g}}{\phi^{\g}(y)} \\
                           &\geq (4/5)^{1+\g} \frac{1}{2}\frac{[u_{\phi}^-(y)-u_{\phi}^-(x)]^{1+\g}}{\phi^{\g}(y)} \\
                           &\geq  \frac{2}{5}(4/5)^{\g}\frac{[u_{\phi}^-(y)-u_{\phi}^-(x)]^{1+\g}}{\phi^{\g}(y)}.
     \end{aligned}
    \]
  \end{proof}

\begin{proposition}  \label{p:gamma}
  Let $F$ be a function satisfying $F,F'' \geq 0$ for $x \geq 0$. Assume also $F(0)=0$. If $y \geq x \geq 0$, then 
   \[
    F(y)-F(x) \geq F(y-x).
   \]
 \end{proposition}
 
 \begin{proof}
  For fixed $h>0$, 
   \[
    \frac{d}{dx} \frac{F(x+h)-F(x)}{h} = \frac{F'(x+h)-F'(x)}{h} \geq 0.
   \]
  Then for $x,h\geq 0$
   \[
    \frac{F(x+h)-F(x)}{h} \geq \frac{F(0+h)-F(0)}{h} = \frac{F(h)}{h}.
   \]
  Let $h=y-x$, and multiply both sides of the equation by $y-x$. 
 \end{proof}

 \begin{proposition} \label{p:exp}
  Let $0<\gamma<1$. Let $x,d\geq 0$. Then 
   \[
    (x+d)^{\g} -d^{\g} \leq 2^{\g} x^{\g}
   \]
 \end{proposition}

 \begin{proof}
  First assume $x\leq d$. From the concavity of $x^{\g}$ we have
   \[
    (x+d)^{\g}-d^{\g}\leq \g d^{\g-1}x = \g d^{\g-1}x^{1-\g}x^{\g} \leq \g x^{\g}.
   \]
  If on the other hand $x >d$, then
   \[
    (x+d)^{\g}-d^{\g}\leq (x+d)^{\g}\leq (2x)^{\g}.
   \]
 \end{proof}

\bibliographystyle{amsplain}
\bibliography{refpme}

\providecommand{\bysame}{\leavevmode\hbox to3em{\hrulefill}\thinspace}
\providecommand{\MR}{\relax\ifhmode\unskip\space\fi MR }
\providecommand{\MRhref}[2]{%
  \href{http://www.ams.org/mathscinet-getitem?mr=#1}{#2}
}
\providecommand{\href}[2]{#2}
\begin{thebibliography}{10}

\bibitem{acv15}
Mark Allen, Luis Caffarelli, and Alexis Vasseur, \emph{A parabolic problem with
  a fractional time derivative}, preprint available on arxiv.org, 2015.

\bibitem{bmst15}
A~Bernardis, F.J. Mart´in-Reyes, P.R. Stinga, and J.~Torrea, \emph{Maximum
  principles, extension problem and inversion for nonlocal one-sided
  equations}, preprint available on arxiv.org, 2015.

\bibitem{ccv11}
Luis Caffarelli, Chi~Hin Chan, and Alexis Vasseur, \emph{Regularity theory for
  parabolic nonlinear integral operators}, J. Amer. Math. Soc. \textbf{24}
  (2011), no.~3, 849--869. \MR{2784330 (2012c:45024)}

\bibitem{cfv13}
Luis Caffarelli, Fernando Soria, and Juan~Luis V{\'a}zquez, \emph{Regularity of
  solutions of the fractional porous medium flow}, J. Eur. Math. Soc. (JEMS)
  \textbf{15} (2013), no.~5, 1701--1746. \MR{3082241}

\bibitem{cv11}
Luis Caffarelli and Juan~Luis Vazquez, \emph{Nonlinear porous medium flow with
  fractional potential pressure}, Arch. Ration. Mech. Anal. \textbf{202}
  (2011), no.~2, 537--565. \MR{2847534 (2012j:76096)}

\bibitem{c99}
Michele Caputo, \emph{Diffusion of fluids in porous media with memory},
  Geothermics \textbf{28} (1999), no.~1, 113--130.

\bibitem{d04}
D.~del Castillo-Negrete, B.A. Carreras, and V.E. Lynch, \emph{Fractional
  diffusion in plasma turbulence}, Physics of Plasmas (2004).

\bibitem{d05}
\bysame, \emph{Nondiffusive transport in plasma turbulene: A fractional
  diffusion approach}, Physical Review Letters (2005).

\bibitem{dpv12}
Eleonora Di~Nezza, Giampiero Palatucci, and Enrico Valdinoci,
  \emph{Hitchhiker's guide to the fractional {S}obolev spaces}, Bull. Sci.
  Math. \textbf{136} (2012), no.~5, 521--573. \MR{2944369}

\bibitem{d10}
Kai Diethelm, \emph{The analysis of fractional differential equations}, Lecture
  Notes in Mathematics, vol. 2004, Springer-Verlag, Berlin, 2010, An
  application-oriented exposition using differential operators of Caputo type.
  \MR{2680847 (2011j:34005)}

\bibitem{gt01}
David Gilbarg and Neil~S. Trudinger, \emph{Elliptic partial differential
  equations of second order}, Classics in Mathematics, Springer-Verlag, Berlin,
  2001, Reprint of the 1998 edition. \MR{1814364 (2001k:35004)}

\bibitem{mk00}
Ralf Metzler and Joseph Klafter, \emph{The random walk's guide to anomalous
  diffusion: a fractional dynamics approach}, Phys. Rep. \textbf{339} (2000),
  no.~1, 77. \MR{1809268 (2001k:82082)}

\bibitem{t10}
Hans Triebel, \emph{Theory of function spaces}, Modern Birkh\"auser Classics,
  Birkh\"auser/Springer Basel AG, Basel, 2010, Reprint of 1983 edition
  [MR0730762], Also published in 1983 by Birkh{\"a}user Verlag [MR0781540].
  \MR{3024598}

\bibitem{v07}
Juan~Luis V{\'a}zquez, \emph{The porous medium equation}, Oxford Mathematical
  Monographs, The Clarendon Press, Oxford University Press, Oxford, 2007,
  Mathematical theory. \MR{2286292 (2008e:35003)}

\bibitem{z12}
Rico Zacher, \emph{Global strong solvability of a quasilinear subdiffusion
  problem}, J. Evol. Equ. \textbf{12} (2012), no.~4, 813--831. \MR{3000457}

\bibitem{z02}
G.~M. Zaslavsky, \emph{Chaos, fractional kinetics, and anomalous transport},
  Phys. Rep. \textbf{371} (2002), no.~6, 461--580. \MR{1937584 (2003i:70030)}

\bibitem{zxc14}
Xuhuan Zhou, Weiliang Xiao, and Jiecheng Chen, \emph{Fractional porous medium
  and mean field equations in {B}esov spaces}, Electron. J. Differential
  Equations (2014), No. 199, 14. \MR{3273082}

\end{thebibliography}

\end{document}